\documentclass[11pt]{amsart}
\usepackage{fullpage}

\usepackage{graphicx}
\usepackage{amssymb}
\usepackage{amsmath}
\usepackage{mathtools}
\usepackage[latin1]{inputenc}
\usepackage[english]{babel}
\usepackage{listings}
\usepackage{relsize}
\usepackage{tikz-qtree}

\usepackage{float}
\usepackage{cleveref}
\usepackage{tikz}
\usepackage{dsfont}
\usepackage{color}
\usepackage[normalem]{ulem}

\newcommand{\C}{\mathbb{C}}
\newcommand{\N}{\mathbb{N}}
\newcommand{\R}{\mathbb{R}}

\newcommand{\FF}{\mathbb{F}}

\newcommand{\A}{\mathcal{A}}
\newcommand{\Bc}{\mathcal{B}}
\newcommand{\Kc}{\mathcal{K}}
\newcommand{\LL}{\mathcal{L}}
\newcommand{\Nica}{\mathcal{N}}
\newcommand{\NT}{\mathcal{NT}}
\newcommand{\T}{\mathcal{T}}

\newcommand{\bs}{\backslash}
\newcommand{\ml}{\mathlarger}
\newcommand{\dis}{\bigsqcup}
\newcommand{\dcup}{\: \mathlarger{\sqcup} \:}
\newcommand{\ra}{\rangle}
\newcommand{\la}{\langle}

\newcommand{\inv}{^{-1}}
\newcommand{\Tr}{\operatorname{Tr}}
\newcommand{\KMS}{\operatorname{KMS}}
\def\clq{\operatorname{Cl}}

\def\pmin{P_{\operatorname{inf}}}

\theoremstyle{plain}
\newtheorem{thm}{Theorem}[section]
\newtheorem{lem}[thm]{Lemma}
\newtheorem{prop}[thm]{Proposition}
\newtheorem{corollary}[thm]{Corollary}

\theoremstyle{definition}
\newtheorem{defn}[thm]{Definition}
\newtheorem{exmp}[thm]{Example}
\newtheorem{rem}[thm]{Remark}

\numberwithin{equation}{section}
\newcommand{\thmref}[1]{Theorem~\ref{#1}}
\newcommand{\secref}[1]{Section~\ref{#1}}
\newcommand{\proref}[1]{Proposition~\ref{#1}}
\newcommand{\lemref}[1]{Lemma~\ref{#1}}
\newcommand{\corref}[1]{Corollary~\ref{#1}}

\newcommand{\defref}[1]{Definition~\ref{#1}}

\makeindex
\newcommand{\myemph}[1]{\index{#1}\emph{#1}}

\usetikzlibrary{graphs,graphs.standard}

\begin{document}

\title[KMS states of quasi-free dynamics]{KMS states of quasi-free dynamics on
$C^*$-algebras of product systems over right LCM monoids
}

\author[Gazdag]{Luca Eva Gazdag}
\address{Department of Mathematics, University of Oslo,
P.O. Box 1053 Blindern, N-0316 Oslo, Norway.}
\email{lucaeg@student.matnat.uio.no}
\author[Laca]{Marcelo Laca}
\address{Department of Mathematics and Statistics \\
 University of Victoria\\
 Victoria BC V8W 2Y2, Canada}
\email{laca@uvic.ca}
\thanks{M. Laca  (ORCiD: 0000-0002-0901-8165).}
\author[Larsen]{Nadia S. Larsen}
\address{Department of Mathematics, University of Oslo,
P.O. Box 1053 Blindern, N-0316 Oslo, Norway.}
\email{nadiasl@math.uio.no}

\begin{abstract}
We generalise recent results of Afsar, Larsen and Neshveyev for product systems over quasi-lattice orders  by showing  that the equilibrium states of  quasi-free dynamics on the Nica-Toeplitz $C^*$-algebras of  product systems over right LCM monoids must satisfy a positivity condition encoded in a system of inequalities satisfied by their restrictions to the coefficient algebra.  We prove that the reduction of this positivity condition to a finite subset of inequalities is valid for a wider class of monoids that properly includes finite-type Artin monoids, answering  a question left open in their work.
Our main technical tool is a combinatorially generated tree modelled on a recent construction developed by Boyu Li for dilations of contractive representations. We also obtain a reduction of the positivity condition to inequalities arising from a certain minimal subset that may not be finite but has the advantage of holding for all Noetherian right LCM monoids, and we present an example, arising from a finite-type Artin monoid, that exhibits a gap in its inverse temperature space.
\end{abstract}

\date{July 29, 2022}

\maketitle

\section{Introduction}
 A very rich context for the study of KMS states  is provided by  the $C^*$-algebras $\T(X)$ and $\mathcal{O}(X)$ associated to a $C^*$-correspondence $X$ by Pimsner \cite{Pi97}.  These algebras can be endowed with various quasi-free dynamics, and one of the key features of the resulting $C^*$-dynamical systems is the relation between their KMS states and the traces of the coefficient algebra \cite{LN04}. Since the $C^*$-algebras associated to product systems of $C^*$-correspondences are natural generalisations of Pimsner algebras \cite{F02,SY10}, it is natural to wonder about a similar study for product systems.

Along these lines,  KMS states of Nica-Toeplitz $C^*$-algebras and  their Cuntz-Pimsner quotients   for special cases of product systems have been studied in \cite{HLS,AaHR18,BLRS19}.
More recently, partly motivated by  the connection between combinatorial and analytic properties of quasi-lattice ordered monoids initially explored in \cite{BLRS19}, Afsar, Larsen and Neshveyev  have completely characterised the traces of the coefficient algebra of a product system $X$ that extend to KMS states of $\NT(X)$ for all compactly aligned product systems over quasi-lattice ordered semigroups \cite{ALN18}.
Their characterisation provides a unified perspective of several special cases previously studied and  consists of a system of inequalities that amount to a positivity condition for induced traces. In principle such a system involves infinitely many inequalities. But for product systems over right-angled Artin monoids $A_M^+$, the positivity condition reduces to a more tractable system of inequalities that depends only on the canonical generating set of $A_M^+$ and is finite whenever $A_M^+$ is finitely generated \cite[Section 9]{ALN18}. This immediately raises the concrete question of whether a similar result holds for other Artin monoids or indeed other monoids $P$.

In this article we show that this `reduction of positivity to generators' is valid  for finite-type Artin monoids, as well as other Artin monoids that are neither right-angled nor finite-type. Along the way, we also obtain a  simplified characterisation, valid  when $P$ is a Noetherian right LCM monoid, in which the inequalities arise from  a certain minimal subset $\pmin$ of $P$ that was recently introduced in  \cite{BL18} by B. Li in the context of regular dilations of contractive representations of monoids.

In order to describe our results  in more precise terms, it is helpful to recall the abstract characterisation of KMS states in terms of the systems of inequalities from \cite{ALN18}. We refer to  \secref{sec:background} below for the necessary definitions.

 Suppose that $(G,P)$ is a quasi-lattice ordered group, that $X=\{X_p\}_{p \in P}$ is a compactly aligned product system of $C^*$-correspondences over $P$ with $X_e = A$, that  the $C^*$-algebra $\NT (X)$ is endowed with a time evolution defined by a homomorphism $N: P \to (0,\infty)$ and that $\beta \in \R$. For a finite subset $K$ of $P$ we denote by $\lor K \in P \cup \{\infty \}$ the smallest upper bound of the elements in $K$.
  If $\tau$ is a tracial state on $A$ and $p\in P$, then $\Tr_{\tau}^p$ denotes the trace on $A$ induced from  $\tau$ through the bimodule $X_p$ as in \cite[Theorem 1.1]{LN04}.
  This is the set-up of  \cite[Theorem 2.1]{ALN18}, which shows that the inequality
 \begin{equation}\label{eqn:2.1}
\tau(a) + \sum_{\emptyset \neq K \subset J} (-1)^{|K|}N(\lor K)^{-\beta}
\Tr_{\tau}^{ \lor K}(a)
\geq 0, \qquad  a\in A_+, \quad J \subset P\bs \{e\},
\end{equation}
must hold for the restriction $\tau = \phi|_{A}$ of a KMS$_\beta$ state $\phi$ on $\NT(X)$ for every finite  $J$. By convention, the summands corresponding to
$\lor K = \infty$ are zero. Moreover every tracial state of $A$ satisfying \eqref{eqn:2.1} is the restriction of a $\KMS_{\beta}$-state $\phi$ on $\Nica \T (X)$. In general, the restriction map may not be one to one, but
\cite[Theorem 5.1]{ALN18} shows that the gauge invariant $\KMS_{\beta}$-states are in affine bijection with the tracial states of $A$ satisfying the positivity condition \eqref{eqn:2.1}.

Inequalities of the form \eqref{eqn:2.1} are often referred to as  subinvariance conditions and have appeared before in the study of KMS states: the case of a single correspondence $X$ requires one inequality \cite[Theorem 2.1]{LN04}; the case of a product system  arising from finitely many $*$-commuting local homeomorphisms of a compact space involves finitely many inequalities \cite[Proposition 4.1]{AaHR18}.

In principle, by \cite[Theorem 5.1]{ALN18}, the gauge-invariant $\KMS_{\beta}$-states of the dynamics on  $\NT(X)$ arising from a homomorphism $N$ can be determined by finding all the traces on $A$ that satisfy the positivity condition. Needless to say, this procedure may be unmanageable in practice, so it becomes important to identify cases in which a smaller system of inequalities suffices.

 Such a simplification was obtained in \cite[Theorem 9.1]{ALN18} for the case in which $P$ is a right-angled Artin monoid $A_M^+$. Specifically,  denoting the set of canonical generators of $A_M^+$ by $S$, in this case  the positivity condition \eqref{eqn:2.1} is equivalent to
\begin{equation}\label{eq:ineq subset of S}
\tau(a) + \sum_{\substack{ \emptyset \neq K \subset J \\ K \in cl(S)}}
(-1)^{|K|}N(\lor K)^{-\beta}\Tr_{\tau}^{ \lor K}(a) \geq 0
 \text{ for all  finite }J \subset S\text{ and }a\in A_+,
\end{equation}
where the sum is over the collection  $cl(S)$ of cliques of $S$, namely the collection of subsets of $S$  that have an upper bound in $P$.  We refer to such a result as reduction of the positivity condition \eqref{eqn:2.1} to atoms.  Our main results are various reductions of the positivity condition to smaller systems of inequalities
valid for several classes of monoids $P$.

As for our context, we begin by considering general right LCM monoids, which are the left cancellative monoids that  have least common upper bounds and may contain nontrivial invertible elements. For some results we focus on the more restrictive class of group embeddable right LCM monoids that have no  invertible elements other than the identity;
 these are known as {\em weak quasi-lattice ordered monoids}  \cite[Definition 2.1]{ABCD19}.  In most situations we will also assume $P$ to be Noetherian; this is a suitable assumption that ensures the existence of a set of atoms which, together with the invertible elements, generate the monoid. It is known, for example, that the existence of a length function implies Noetherianity.  Our main applications and most definite results are for finite-type Artin monoids, for which we answer in the positive, the question of reduction of positivity to generators raised in \cite{ALN18}.

The paper is organised as follows. In \secref{sec:background} we review briefly the necessary background on monoids. The main results of \secref{sec:prodsyst} are the generalisations of Theorem 2.1 and Theorem 5.1 of \cite {ALN18} to the Nica-Toeplitz algebra $\NT(X)$ of a compactly aligned product system $X=(X_p)_{p\in P}$ over  a right LCM monoid $P$. This is fairly straightforward, except  for the need to upgrade several results about inducing in stages  for product systems over Noetherian monoids $P$ that have nontrivial invertible elements.
 This widens the range of applications and may be of independent interest.
Since we do not assume the existence of an ambient group, we also propose a natural notion of gauge-invariance for KMS states on $\NT(X)$. The key is to use certain conditional expectations that were implicit in \cite{KL1} and \cite{KL2} to replace the conditional expectations of the dual coactions that exist only when $P$ embeds in a group.

In \secref{sec:tree}, motivated by \cite{BL18}, we construct a tree by an induction procedure that provides a recursion formula to keep track of how the positivity condition for a given level depends on the next level, \proref{prop:recursion lists semilattice}. Our first main result is \thmref{thm:finite type} where we show that reduction of positivity holds whenever the tree is finite. As applications we prove reduction results for finite-type Artin monoids, \corref{cor:directedcase}, and we also recover in  \proref{pro:finitetree-RAAMcase} the reduction result for right-angled Artin monoids, which was proved by different means in \cite[Theorem 9.1]{ALN18}.  Further, \corref{cor:simplification} spells out the main practical method to find KMS states by verifying positivity on suitable finite subsets of $P$.

In \secref{sec:reduction minimal set all qlo} we take a different approach to reduction of the positivity condition. More precisely, we combine the technique from \cite{ALN18} for proving reduction for right-angled Artin monoids with a minimal set introduced by Boyu Li in  \cite{BL18}. This minimal set is the smallest subset of $P$ that contains the atoms and is closed under taking left divisors of least upper bounds.
The key step, \lemref{lem:int}, is to show that a certain algebra of subsets of $P$ can be constructed from ideals $pP \subset P$ with $p$ in the minimal set.
 Our second main result is \thmref{thm:Pmin}, which shows that for all Noetherian  weak quasi-lattice orders  there is a reduction of positivity to a subsystem of inequalities corresponding to the minimal subset.

In \secref{sec: gaps} we specialise to product systems with one-dimensional fibres, that is,  $X_p=\mathbb{C}$ for all $p\in P$, and we illustrate, for the $C^*$-algebras of Artin braid monoids $B_n^+$ for $n\geq 3$, the technique of finding KMS states by solving the system of equations (explicitly for $n=4$).

 Finally, in \secref{sec:alt} we refine the search of a suitable algebra of subsets of $P$ for which reduction of positivity holds by going all the way down to the set of atoms inside the minimal subset of $P$. \thmref{thm:redifalg} formalises the result, and its main application, \corref{thm:NewRed}, provides a reduction of the positivity condition to atoms for  a class of Artin monoids, obtained through free and direct products,  which properly contains the finite-type and the right-angled ones. Explicit examples are given in  \proref{prop:B3alg}.  It remains an open question whether reduction of positivity to atoms holds for all Artin monoids.

\section{Background}\label{sec:background}
\subsection{Semigroups, conditional right LCMs, and noetherianity.}
All semigroups in this paper have an identity, denoted $e$, and thus are monoids. We concentrate our attention on left cancellative monoids that admit a partial order under which conditional right least common multiples exist, also referred to as \emph{right LCM monoids}. These objects fall under a broader class of monoids, and we refer to the monograph \cite{Garside-book} as reference.

For $p,q$ in a monoid $P$ we say that $q$ is a \emph{right multiple} of $p$, written $p\leq q$, provided that there is $q'\in P$ such that $pq'=q$. Equivalently, $q\in pP$. A further  equivalent formulation is that $p$ is a \emph{left divisor} of $q$. The set of left-divisors of $q$ will be denoted $\operatorname{Div}_l(q)$.  The relation $\leq$ is an order when $P$ is left cancellative with no non-trivial invertible elements. There are obvious definitions of the symmetric notions of left multiple (right divisor). The invertible elements in $P$ form a group $P^*$, possibly the trivial group $\{e\}$.

\begin{defn} \cite[Definition 2.9]{Garside-book}
Let $P$ be a left-cancellative monoid and let $U$ be a subset of $P$. The element $r$ of $P$ is a least common right multiple, or \emph{right LCM}, of $U$ if $r$ is a right multiple of each element of $U$ and every element $r'$ of $P$ that is a right multiple of each element of $U$ is a right multiple of $r$.
\end{defn}

A subset $U$ of $P$ that admits a right LCM $r\in P$ is called a \emph{clique} in $P$. We denote a right LCM of $U$ by $\vee U$ and note that $(\vee U) u$ is a right LCM of $U$ for every $u\in P^*$. If $U$ is a clique consisting of two elements $p$ and $q$, we write $p\vee q$ for a right LCM of $\{p,q\}$. We write $\vee U=\infty$ when $U$ has no right LCM. By convention, the empty set is a clique. Obviously, every one-element subset of $P$ is a clique. For a subset $J\subset P$ we denote the collection of cliques $U$ of $J$ by $cl(J)$.

Note that the operation $\vee$ is associative on $P$ (since $P$ is left-cancellative), a fact which will be used tacitly throughout.

We recall now an important class of monoids: in \cite{Garside-book}, their defining property is called existence of  conditional right LCM's.  Right-cancellative monoids that admit conditional lcm's with respect to left multiples have been studied in e.g. \cite{Law}.

\begin{defn} A left-cancellative monoid $P$ with the property that every two elements $p$ and $q$ that admit a common
right multiple admit a right LCM will be called a right LCM monoid.
\end{defn}

A monoid $P$ is \emph{left-Noetherian} (respectively \emph{right-Noetherian}, respectively \emph{Noetherian}) provided that there exists no infinite descending sequence in $P$ with respect to proper left-divisibility (respectively proper right-divisibility, respectively proper factor relation). By \cite[Proposition II.2.29]{Garside-book}, a left-cancellative monoid is Noetherian if and only if it is both right- and left-Noetherian.

If $P$ is a left-cancellative monoid that is right-Noetherian, and if $q\in P$, then every sequence in $\operatorname{Div}_l(q)$ of the form $p_1,p_1p_2, p_1p_2p_3,\dots$ with $p_n$ non-invertible is necessarily finite, cf. \cite[Proposition II.2.28]{Garside-book} (i.e. The divisors of $q$ have no infinite increasing  sequences.)
   An analogous argument shows that if $P$ is also right-cancellative and left-Noetherian, then for each $q\in P$, any sequence of the form $r_1,r_2r_1,r_3r_2r_1,\dots$ with $r_n$ non-invertible (i.e. an increasing sequence with respect to right-divisibility) in the set $\operatorname{Div}_r(q)$ of right-divisors of $q$ is finite.

 An \emph{atom} in a monoid $P$ is an element $s \in P\setminus P^*$ such that every decomposition of $s$ as a finite product of elements in $P$ contains at most one element in  $P\setminus P^*$, see \cite[Ch. II, Definition 2.52]{Garside-book}. We denote the set of atoms in $P$ by $P_a$.

It is proved in \cite[Corollary II.2.59]{Garside-book} that in a Noetherian left-cancellative monoid $P$ that  contains no non-trivial invertible elements, the set of atoms $P_a$ is the smallest subset generating $P$. In an arbitrary left-cancellative Noetherian monoid, a generating set $P_{\operatorname{gen}}$ is any subfamily that generates $P^*$ and contains at least one element in each equivalence class of atoms for the relation $\sim$ with $s\sim s'$ if and only if $s'=su$ for $u\in P^*$, \cite[Proposition 2.58]{Garside-book}.

We recall from \cite{N92}, see also \cite{CL02} that a discrete group $G$ and  a subsemigroup $P$ so that $P\cap P^{-1}=\{e\}$ form a quasi-lattice ordered pair for the partial order $x\leq y  \Leftrightarrow x^{-1}y\in P$, with $x,y\in G$, if the following two conditions hold:
\begin{enumerate}
\item[(QL1)] For all $p,q\in P$ such that $(pq^{-1})P\cap P\neq \emptyset$ there is (a necessarily unique) $r\in P$ so that $(pq^{-1})P\cap P=rP$, and
\item [(QL2)] Given any pair $p,q\in P$ so that $pP\cap qP\neq \emptyset$   there is (a necessarily unique) element $p\vee q$ in $P$, their \emph{least common upper bound}, so that $pP\cap qP=(p\vee q)P$.
\end{enumerate}
A  discrete group $G$ and  a subsemigroup $P$ so that $P\cap P^{-1}=\{e\}$ form a weak quasi-lattice ordered pair if condition (QL2) alone is satisfied.

\subsection{Artin monoids}
We now recall some basic facts about Artin monoids. A \myemph{Coxeter matrix} over a set $S$ is a matrix $M=(m_{st})_{s,t \in S}$  such that
$m_{st}=m_{ts} \in \{ 2, \dots, \infty\}$
for all $s \neq t$ and $m_{ss} = 1$ for all $s,t \in S$, with the convention that $M$ may be an infinite matrix.
The \emph{Artin group }$A_M$ associated to $M$ is the group with
generating set $S $ and presentation
\begin{align*}
\{ S \: | \: \la st\ra^{m_{st}} = \la ts\ra^{m_{ts}} \:
\text{for all } s,t \in S \},
\end{align*}
where $\la st \ra^{m_{st}} = sts \dots$ is the alternating product of length $m_{st}$; notice that the last letter of $\la st \ra^{m_{st}}$  is $t$  when $m_{st}$ is even and $s$ when $m_{st}$ is odd.
 The \myemph{Artin monoid} $A_M^+$ is the monoid with the same presentation. As a consequence of \cite[Verk\"urzungslemma]{BS72},  every Artin monoid admits conditional right LCM's and the explicit formula for the least common upper bound of two generators is given by the obvious guess.

\begin{lem} \label{bound}
Every Artin monoid  $A_M^+$ embeds in the corresponding Artin group $A_M$, and the pair $(A_M,A_M^+)$ is  weak quasi-lattice ordered.
If $s$ and $t$ are generators, then
$s \lor t = \la st \ra^{m_{s,t}} = \la ts \ra^{m_{t,s}}$ for all $s,t \in S$ with $m_{s,t} = m_{t,s} <\infty$.
\end{lem}

\begin{proof} The general embedding result is from \cite{Paris2002}, although various special cases had been known previously. That the pair is weak quasi-lattice ordered is a consequence of \cite[Verk\"urzungslemma]{BS72}.
For the last assertion, it is clear that $ \la st \ra^{m_{s,t}} =\la ts \ra^{m_{t,s}}$ is an
upper bound for  both $s$ and $t$ in $S$. Suppose $x \in P$ is an arbitrary upper
bound for $s$ and $t$, that is, there exist $p, q \in P$ such that $x = s p = t q$.
By \cite[Verk\"urzungslemma]{BS72}, there exists a positive word $w$ such that  $p= \la ts \ra ^
 {m_{s,t}-1}w$ and $q= \la st\ra^{m_{t,s}-1}w$. Hence $x = s p = \la s t\ra^{m_{s,t}}w$, and hence $\la st \ra^{m_{s,t}} \leq x$. Since $x$ is arbitrary we can conclude that
$ s \lor t = \la st\ra^{m_{s,t}} = \la ts \ra^{m_{ts}}$.
\end{proof}

There are two particularly well understood classes of Artin groups and monoids: \emph{Right-angled Artin groups}, where $m_{ij} \in \{2,\infty\}$ for all $i\neq j$ in $\Lambda$,
and \emph{finite-type Artin groups}, which are those for which the corresponding \emph{Coxeter group}, which is the group obtained by adding the relations
$s_i^2 =e$ to those of $A_M$,
 is finite, from this it  follows that their Coxeter matrix is finite and that it contains no infinities. Both form quasi-lattice ordered pairs, cf.  \cite{BS72} and \cite{CL02}.

\begin{rem}\label{rem:length}
Since Coxeter matrices are symmetric,  the  relations defining Artin monoids are homogeneous, in the sense that they preserve the total number of generators in a word. This implies that the usual length function on the free monoid factors through Artin monoids. As a consequence, all Artin monoids are Noetherian.
\end{rem}

\subsection{An algebra of subsets of $P$}
Suppose that $P$ is a right LCM monoid. For each $p\in P$, let $1_{pP}$ denote the characteristic function  of the set
$\{q\in P:p\leq q\}$.
Then  $B_P=\overline{\operatorname{span}}\{1_{pP}: p\in P\}$ is a commutative $C^*$-subalgebra of
$l^\infty(P)$, see \cite[Section 2]{KL2} for further properties.

Following \cite[Section 2]{ALN18} we denote by $\Bc_P$ the algebra of subsets of $P$ generated by the sets $pP$ for $p\in P$; by an algebra we mean a nonempty collection of subsets of $P$ that is closed under taking complements and finite unions. Then the set of projections in $B_P$ is precisely the set of characteristic functions of the elements of $\Bc_P$.
 For easy reference we state a version of \cite[Lemma 2.4]{ALN18} valid for right LCM monoids.

\begin{lem}\textup{(cf. \cite[Lemma 2.4]{ALN18})}\label{lem:algebra of sets general}
For each finite subset $J$ of $P$ let
\begin{equation}\label{eq:omega-set}
\Omega_J:=\bigcap_{q\in J}(P\setminus qP),
\end{equation}
where by convention $\Omega_\emptyset = P$ and $\Omega_J=\emptyset$ whenever $J\cap P^*\neq \emptyset$.
Then every set $\Bc_P$ is a finite disjoint union of sets of the form $p\Omega_J$, where
$p \in P$ and $J \subset P$ is a  finite set.
\end{lem}
%If $J=\{q\}$, then we  write $\Omega_q$ for $\Omega_{\{q\}}$.
In the reduction of the positivity condition to generators proved in \cite[Theorem 9.1]{ALN18} in the case of {right-angled} Artin monoids, the key ingredient is
that to generate $\Bc_P$ it suffices to take subsets $J \subset S$ of generators, see \cite[Lemma 9.3]{ALN18}.

\section{Product systems over right LCM monoids}\label{sec:prodsyst}
Suppose that $P$ is a left cancellative monoid. The notion of a product system $X$ over $P$ of $C^*$-correspondences was introduced in \cite{F02}. In particular, \cite{F02} identified the important class of Nica-Toeplitz covariant representations of a compactly aligned product system $X$ over a monoid $P$ in a quasi-lattice ordered pair $(G, P)$.

The analysis of KMS states for time evolutions on  Nica-Toeplitz $C^*$-algebras of compactly aligned product systems over positive cones of quasi-lattice ordered groups was carried out in  \cite{ALN18}.
One of our aims here is to extend  this analysis to product systems over the positive cones of weak quasi-lattice ordered groups, and, more generally,  over right LCM monoids. Another more specific aim is to show that the reduction of the characterisation of KMS states to generators holds in other cases as well, notably for finite-type Artin monoids.
We begin by briefly recalling some necessary background. For a comprehensive discussion of product systems of $C^*$-correspondences see \cite{F02, SY10, CLSV11} in the quasi-lattice ordered case and the generalisation \cite{BLS2, KL2} to the right LCM case.

Suppose that  $P$ is a (countable) cancellative semigroup with identity $e$ and $A$  a $C^*$-algebra.
Let $X$ be a right $C^*$-Hilbert $A$-module. We denote the adjointable operators on $X$ by $\LL (X)$ with identity $I_X$.
A $C^*$-correspondence $X$ over $A$ is a right $C^*$-Hilbert $A$-module
equipped with a left action given by a $*$-homomorphism $\varphi:A \to \LL (X)$. The correspondence is  \myemph{essential} if $\overline{\varphi(A)(X)} = X$. A collection $(X_p)_{p \in P}$ of  $C^*$-correspondences over
$A$ is a \myemph{product system}, cf. \cite{F02}, if the following properties are satisfied:
\begin{itemize}
	\item[$(i)$] $\bigcup_{p \in P} X_p$ has a monoid structure such that for any $\xi \in X_p$ and $\zeta \in X_q$ we have
	$\xi \zeta \in X_{pq}$ and the map $\xi \otimes \zeta \mapsto \xi \zeta$ extends to an isometric isomorphism $X_p \otimes_A X_q \simeq X_{pq}$ for all $p,q \in P$ with $p\neq e$.
	\item[$(ii)$] $X_e$ is the canonical correspondence ${_{A}}A_{A}$ and the product maps $X_e \times X_p \to X_p$ and
	$X_p \times X_e \to X_p$ coincide with the structure maps of the $A$-bimodules $X_p$ for all $p \in P$.
\end{itemize}	

% As in  \cite{ALN18}, all product systems in this paper are assumed to have essential fibres $X_p$ for all $p\in P$.

\subsection{The Nica-Toeplitz $C^*$-algebra of $X$} A universal $C^*$-algebra $\mathcal{T}(X)$ is associated to a product system $X$ via representations $\psi$ of $X$. We refer to \cite{F02} for the definition of an arbitrary  representation $\psi:X\to B$ with $B$ a $C^*$-algebra. Next we recall in some detail the notion of Nica covariant representation of $X$, under further assumptions on $P$ and $X$.

Suppose that $P$ is a right LCM monoid and $X$ is a product system over $P$ of $C^*$-correspondences over a $C^*$-algebra $A$ such that $X_p$ is essential for every $p\in P$. Note that if $P^*\neq \{e\}$ then $X$ will automatically have essential fibres, cf. \cite[Remark 1.3]{KL2}. The requirement of essential fibres is not needed to define Nica covariant representations and the universal $C^*$-algebra $\NT(X)$, but is useful when we also bring into the discussion the reduced Nica-Toeplitz $C^*$-algebra $\NT^r(X)$, cf. also \cite{ALN18}.

Let $I_p$ denote the identity operator in $\mathcal{L}(X_p)$ for every $p\in P$. If $p, r \in P$ with $p\le r$, then  there is a natural $*$-homomorphism $\iota^r_p\colon \LL(X_p)\to\LL(X_r)$ obtained by identifying $X_r$ with $X_p\otimes_A X_{p^{-1}r}$ and mapping $S\in\LL(X_p)$ into $S\otimes I_{p^{-1}r}$. Let $\Kc(X_p)$ denote the ideal of generalised compact operators in $\LL(X_p)$ for each $p$. We say that $X$ is \emph{compactly aligned} if for all $p,q,r\in P$ with $pP\cap qP=rP$ we have that
$$
\iota^{r}_p(\Kc(X_p))\iota^{r}_q(\Kc(X_q))\subset\Kc(X_{r});
$$
no condition is imposed when $pP\cap qP=\emptyset$.
We define elements $\theta_{\xi,\zeta}\in\Kc(X_p)$ by $\theta_{\xi,\zeta}\eta=\xi\langle\zeta,\eta\rangle$, for $\xi,\zeta,\eta\in X_p$. Given a representation $\psi$ of $X$ into $B$, there are $*$-homomorphisms $\psi^{(p)}\colon\Kc(X_p)\to B$  given by $\psi^{(p)}(\theta_{\xi,\zeta})=\psi_p(\xi)\psi_p(\zeta)^*$ for each $p\in P$. A representation $\psi$ is  Nica covariant if
$$
\displaystyle \psi^{(p)}(S)\psi^{(q)}(T) =
\begin{cases}
\psi^{(r)}\big(\iota^{r}_p(S)\iota^{r}_q(T)\big)
& \text{if $pP\cap qP=rP$} \\
0 &\text{otherwise}
\end{cases}
$$
for all $S \in \Kc(X_p)$ and $T \in \Kc(X_q)$. We refer to \cite[Section 6]{BLS2} and \cite[Lemma 2.4]{KL2} for further details.
By definition, the \emph{Nica-Toeplitz algebra} of $X$ is the $C^*$-algebra $\NT(X)$  generated by a universal  Nica covariant representation $i_X\colon X\to \NT(X)$.

Following \cite{ALN18}, for  each multiplicative semigroup homomorphism $N: P \to (0,\infty)$ we will consider the  time evolution $\sigma$ on $\NT(X)$  such that
$\sigma_t(i_X(\xi_p)) = N(p)^{it}i_X(\xi_p)$ for $\xi_p \in X_p$ and $t\in \R$.
We will need to assume throughout that $N$ is trivial on $P^*$, that is, $N_u=1$ for every $u\in P^*$; this is automatic if the range of $N$ is contained in $[1,\infty)$.
Recall that a state $\phi$ of $\NT(X)$  is a KMS state at inverse temperature $\beta$ (a $\KMS_{\beta}$ state) provided that
\[
\phi(xy) = \phi(y\sigma_{i\beta}(x))
\]
for all $x,y$ with $x$ analytic with respect to $\sigma$, that is to say, $t \mapsto \sigma_t(x)$ extends to a C*-algebra-valued entire function. Since in the present case it  suffices to consider the $\sigma$-invariant set of analytic elements  $i_X(\xi_p) i_X(\eta_q)^* $ with $\xi_p\in X_p$ and $\eta_q \in X_q$ because it has  a dense linear span (it suffices to know that $i_X(\xi)$ generate $\NT(X)$ as a $C^*$-algebra, see \cite[Lemma 1.9]{ALN18}), we will see below that
$\KMS_{\beta}$ states of $(\NT(X), \sigma)$ are characterized by the condition
\[
\phi(i_X(\xi_p) i_X(\eta_q)^*) = \delta_{p,q} N^{-\beta}\Tr_\tau^p (\theta_{\eta_q,\xi_p}), \qquad \xi_p\in X_p, \ \ \eta_q \in X_q,
\]
involving the induced trace of  the tracial state $\tau := \phi\vert_A$.

\subsection{Inducing traces via  $C^*$-correspondences} We recall from \cite[Sections 5.1 and 5.2]{P79} that a weight $\psi$ on a $C^*$-algebra $A$ is a map $\psi:A_+\to [0,\infty]$ that satisfies positive homogeneity: $\psi(\alpha a)=\alpha\psi(a)$ for all $\alpha\in \R_+$ and $a\in A_+$, and {additivity:} $\psi(a+b)=\psi(a)+\psi(b)$ for all $a,b\in A_+$. The weight $\psi$ is lower semicontinous if $\{a\mid \psi(a)\leq c\}$ is closed for all $c\in \R_+$. With $A^\psi_+:=\{a\in A_+\mid \psi(a)<\infty\}$, the space $A^\psi=\operatorname{span}A^\psi_+$ is a hereditary $C^*$-algebra of $A$ on which $\psi$ admits a unique extension as a positive linear functional, still denoted $\psi$. If the weight $\psi$  is invariant under conjugation by unitary elements in the
unitisation $\tilde{A}$, we say that it is a trace, in which case $A^\psi$ is an ideal of $A$. Moreover, if $\psi$ is a {finite trace}, meaning that $A^\psi_+=A_+$, then $\psi(ab)=\psi(ba)$ for all $a,b\in A$.

Suppose that $X$ is a $C^*$-correspondence over  $A$ and $\tau$  a tracial positive linear functional on $A$. We recall from \cite[Theorem 1.1]{LN04} that the \emph{induced trace} $\Tr_{\tau}^{X}$ is defined for  $T \in \mathcal{L}(X)$, $T \geq 0$ by
\begin{equation}\label{eq:induced trace}
\Tr_{\tau}^{X}(T) = \sup_{H\subset X} \sum_{\xi \in H}
 \tau( \la \xi, T \xi \ra),
\end{equation}
where the supremum is taken over all finite subsets $H \subset X$ such that
$\sum_{\xi \in H} \theta_{\xi, \xi} \leq 1$.  Then $\Tr_{\tau}^{X}$ is strictly lower semicontinuous  and satisfies \begin{equation*}\label{eq:induced trace}
\Tr_{\tau}^X(T) = \lim_k \sum_{\xi \in H_k} \tau( \la \xi, T \xi \ra ) \quad \text{for all} \: T \in \LL(X)_+.
\end{equation*}
whenever $H_k$ is  a generalized sequence of finite sets in $X$ such that
$\{ \sum_{\xi \in H_k} \theta_{\xi,\xi} \}_k$ is   an approximate unit in $\mathcal{K}(X)$, moreover, $\Tr_{\tau}^{X}$
extends to a positive semifinite trace on $\LL (X)$.

 We are interested in the restriction of $\Tr_{\tau}^X$  to the image of $A$ inside
$\LL (X)$ given by the left action. Recall from  \cite{LN04} that for each $C^*$-correspondence $X$ over $A$ there is an operator $F_X$ mapping a tracial positive linear functional $\tau$ on $A$  to a possibly infinite positive trace $F_X\tau$ on $A$ defined by
\begin{equation*}
(F_X\tau)(a)=\Tr_\tau^X(a)
\end{equation*}
for $a\in A_+$. Given a product system
 $ \{X_p\}_{p \in P}$ of essential $C^*$-correspondences over $P$, we follow
\cite[Section 1.4]{ALN18} and write $F_p$ instead of $F_{X_p}$ and $\Tr_\tau^p$ instead of $\Tr_\tau^{X_p}$ for all $p\in P$. Let $\varphi_p$ be the left action on $X_p$ and recall that if
$(\Tr_\tau^q\circ \varphi_p)\vert_{A}$ is a finite positive trace on $A$ for some $q\in P$, then we can use induction in stages \cite[Proposition 1.2]{LN04} and  factor
$F_{pq}(\tau)$ as $F_p(F_q(\tau))$ for all $p\in P$. In other words, for $r=pq$ so that $F_{q}(\tau) = \Tr_{\tau}^{q}|_{A_+}$ is finite we have
\begin{equation}\label{eq:Tr-decomp}
\Tr_{\tau}^r(a)=\Tr_{F_{q}(\tau)}^p(a) \text{ for }a\in A_+.
\end{equation}

In general, the induced trace $\Tr_\tau^q$, or rather $\Tr_\tau^q\circ \varphi_p$, need not be a finite trace on $A$. However, we recall from \cite{ALN18} that
{if $\tau$ satisfies condition \eqref{eqn:2.1}
%\sout{if $\psi$ is a $\KMS_{\beta}$-state and we let $\tau = \psi|_{A}$,}
with $J = \{q\}$, then}
 %$\tau(a) - N(q)^{-\beta}\Tr_{\tau}^q(a) \geq 0$ for all $a\in A_+$. This  implies that
\begin{align}\label{eq:induced-trace-from-KMS-is-finite}
\Tr_{\tau}^{q}(a) \leq \tau(a)N(q)^{\beta} < \infty \quad
\text{for all $a \in A_+$.}
\end{align}
In particular,  $\Tr_{\tau}^q$ is a finite trace on $A$ whenever
$\tau$ is the restriction of a $\KMS_{\beta}$-state.

\subsection{A characterisation of KMS states on $\mathcal{NT}(X)$}

We will now extend {\cite[Theorem 2.1]{ALN18}} to the case of Nica-Toeplitz algebras $\mathcal{NT}(X)$ where $X$ is a product system over a right LCM monoid. Before we can state the generalisation we need some preparation about induced traces for product systems over monoids that have nontrivial invertible elements.

Suppose that $X$ is a product system over a right LCM monoid $P$ with $X_e=A$. Recall that $\varphi_p$ is the left action in $X_p$ and $I_p$  the identity operator in $\mathcal{L}(X_p)$, for all $p\in P$. As observed in \cite{BLS2}, for each $u \in P^*$ the map $T\mapsto T\otimes I_{u^{-1}}$ is an adjunction from $X_u$ to  $X_{u^{-1}}$ in the sense of \cite[Definition 2.17]{CCH}, meaning that it induces a natural isomorphism
\[
\mathcal{L}(Z\otimes_A X_u, Y)\to \mathcal{L}(Z,Y\otimes_A X_{u^{-1}})
\]
for any Hilbert $A$-modules $Y,Z$. One consequence is that by \cite[Theorems 4.4(2) and 4.13]{KPW}, see also \cite[Theorem 2.24]{CCH}, it necessarily must hold that $\varphi_u(a)\subseteq \mathcal{K}(X_u)$ for $u\in P^*$ and $a\in A_+$. A second consequence of the existence of this adjunction is that its restriction to the generalised compact operators is an isomorphism
\[
\mathcal{K}(Z\otimes_A X_u, Y)\to \mathcal{K}(Z,Y\otimes_A X_{u^{-1}}),
\]
cf. \cite[Corollary 3.9]{CCH}, in other words it yields  a local adjunction, for all $u\in P^*$. As in \cite{CCH}, we denote elements in $X_{u^{-1}}$ by $\xi^*$, where $\xi$ runs over $X_u$ (in the terminology of \cite{CCH}, the $C^*$-correspondence $X_{u^{-1}}$ is the conjugate of $X_u$). It follows from \cite[Corollary 3.31]{CCH} that $A$ admits a direct sum decomposition
\[
A=A_{u^{-1}}\oplus A_{u^{-1}}^{\perp}
\]
into two-sided ideals
$A_{u^{-1}}=\overline{\operatorname{span}}\{{\langle \xi_1^* ,\xi_2^*\rangle\mid \xi_1^*,\xi_2^*\in X_{u^{-1}}}\}$ and $A_{u^{-1}}^{\perp}=\operatorname{ker}(\varphi_{u})$. In particular, we have that $\varphi_u(A)= \varphi_u(A_{u^{-1}})$.

We now let $F_{p,q}:X_p\otimes_A X_q\to X_{pq}$ denote the subordinate isomorphisms of $C^*$-correspondences for $p,q\in P$. Note that for each $u\in P^*$, the map
\[
F_{u,u^{-1}}^{-1}:X_e\to X_u\otimes_A X_{u^{-1}}
\]
is adjointable with adjoint $F_{u,u^{-1}}$, and satisfies the properties of being a unit for the adjunction corresponding to tensoring by $X_u$ and $X_{u^{-1}}$. By uniqueness of the unit, the map $F_{u,u^{-1}}^{-1}$ is the adjoint of the map $\delta$  from \cite[Lemma 3.18]{CCH}, hence it corresponds to the left action $\varphi_u$ through a canonical isomorphism of $X_u\otimes_A X_{u^{-1}}$ onto $\mathcal{K}(X_u)$ identifying $\xi_1\otimes \xi_2^*$ with $\theta_{\xi_1,\xi_2}$, where $\xi_1,\xi_2\in X_u$, \cite[Proposition 3.29]{CCH}. Moreover, $F_{u,u^{-1}}$ can be identified with the map $\delta$, so
\[
F_{u,u^{-1}}(\xi_1\otimes \xi_2^*)=\langle \xi_1^*,\xi_2^*\rangle\text{ for }\xi_1^*,\xi_2^*\in X_{u^{-1}}.
\]
For the same reason, the map $F_{u^{-1},u}^{-1}:X_e\to X_{u^{-1}}\otimes_A X_u $ must coincide with $\varepsilon$ from \cite[Lemma 3.18]{CCH},
so it corresponds to the left action $\varphi_{{u^{-1}}}$ upon canonical identification of $\xi_1^*\otimes\xi_2$ with $\theta_{\xi_1^*,\xi_2^*}\in \mathcal{K}(X_{u^{-1}})$, \cite[Proposition 3.32]{CCH}. Further,
\[
F_{u^{-1},u}(\xi_1^* \otimes \xi_2)=\langle \xi_1,\xi_2\rangle\text{ for }\xi_1,\xi_2\in X_{u}.
\]
Since $F_{u^{-1},u}^{-1}\circ F_{u,u^{-1}}:X_u\otimes_A X_{u^{-1}}\to X_{u^{-1}}\otimes_A X_u$ maps $\xi\otimes \xi^*$ to $\xi^*\otimes \xi$, where $\xi\in X_u$, it follows that
\begin{equation}\label{eq:flip inner products}
\langle \xi^*,\xi^* \rangle=\langle \xi,\xi\rangle \text{ for all }\xi\in X_{u}.
\end{equation}

\begin{lem}\label{lem:induced traces constant on principal ideals}
Suppose that $X=\{X_p\}_{p \in P}$ is a product system of $C^*$-correspondences over a right LCM monoid $P$ with $X_e=A$.
 Suppose that $\tau$ is a tracial state on $A$. Then the trace on $A$ defined as
$\operatorname{Tr}_\tau^{X_u}(\varphi_u(a))$
for $u\in P^*, a\in A_+$ coincides with $\tau$. In particular, for every $p\in P$ and all $a\in A_+$ we have that \[
\operatorname{Tr}_\tau^{X_{pu}}(\varphi_{pu}(a))=\operatorname{Tr}_\tau^{X_p}(\varphi_p(a))=
\operatorname{Tr}_\tau^{X_{up}}(\varphi_{up}(a)).
\]
\end{lem}

\begin{proof} We recall that $\operatorname{Tr}_\tau^{Y}(\varphi(a))$ is defined in \cite[Proposition 1.2]{LN04} for an arbitrary $C^*$-correspondence $Y$ over $A$. Given $u\in P^*$, let $(e_J)_J$ with $e_J\leq I_u$ be an approximate unit for $\mathcal{K}(X_u)$ where $e_J=\sum_{\eta\in J}\theta_{\eta,\eta}$. Let $a=\langle \xi^*,\xi^*\rangle\in A_+$ for $\xi\in X_{u}$. We have that
\begin{align*}
\operatorname{Tr}_\tau^{X_u}(\varphi_u(a))
&=\operatorname{sup}_J \sum_{\eta\in J} \tau(\langle \eta,\varphi_u(a)\eta\rangle)\\
&=\operatorname{sup}_J \sum_{\eta\in J} \tau(\langle \eta, \langle\xi^*,\xi^*\rangle\eta\rangle) \\
&=\operatorname{sup}_J \sum_{\eta\in J} \tau(\langle \eta, \theta_{\xi,\xi}\eta\rangle)\\
&=\operatorname{sup}_J \sum_{\eta\in J} \tau(\langle \eta, \xi\rangle \langle \xi,\eta\rangle)\\
&=\operatorname{sup}_J \sum_{\eta\in J} \tau(\langle \xi,\eta\rangle\langle \eta, \xi\rangle)\text{ since }\tau\text{ is a trace}\\
&=\operatorname{sup}_J \sum_{\eta\in J} \tau(\langle\xi,\theta_{\eta,\eta}\xi\rangle)\\
&=\tau(\langle \xi,\xi\rangle)\text{ since }e_J\text{ is an approximate unit}\\
&=\tau(a)
\end{align*}
by the assumption on $(e_J)$ and \eqref{eq:flip inner products}. By linearity and continuity we have that $\operatorname{Tr}_\tau^{X_u}(\varphi_u(a))=\tau(a)$ for all $a\in A_+$. The displayed equalities follow by two applications of \cite[Proposition 1.2]{LN04}, namely by letting $q\in P^*$ and, respectively, $p\in P^*$ in equation \eqref{eq:Tr-decomp}.
\end{proof}

\begin{lem}\label{lem:Tr-finite-onA-from-S}
Let  $P$ be a Noetherian right LCM monoid, $\{X_p\}_{p\in P}$ a compactly aligned product system of essential
$C^*$-correspondences over $P$ with $X_e=A$,  $N: P \to (0,\infty)$ a multiplicative homomorphism, $\tau$ a finite positive trace on $A$ and $\beta\in \mathbb{R}$.
\begin{enumerate}
\item[(i)]\label{eq: Tr dominated by multiple of tau}
Suppose that $\Tr_{\tau}^s(a) \leq N(s)^{\beta}\tau(a)$ for all $s\in P_a, \ a\in A_+$. Then $\Tr_{\tau}^p(a) \leq N(p)^{\beta}\tau(a)$ for all $p \in P$ and $a \in A_+$. In particular $\Tr_{\tau}^p\vert_A$ is a tracial  linear functional on $A$ for every $p$.

\item[(ii)]\label{eq:Tr-tau-U-finite}
 Suppose that $\Tr_\tau^p\vert_A$ is finite for every $p\in P$ and for each finite subset $U$ of  $P$ with $\lor U < \infty$  set $\tau_U =\Tr_{\tau}^{\lor U}\vert_A$. Then $\Tr_{\tau_U}^p(a) < \infty$ for all $p \in P$ and $a \in A_+$.
\end{enumerate}
\end{lem}

\begin{proof} Let $N_0(p)=N(p)^{\beta}$ for all $\in P$. Assume the hypothesis of (i) and let $p \in P$. If $p\in P^*$, then $\Tr_{\tau}^p(a)=\tau(a)=N_0(p)\tau(a)$ for all $a\in A_+$ by \lemref{lem:induced traces constant on principal ideals}. If $p\in P\setminus P^*$, then the element  $p$ can be expressed as a product of  the form $p=gs_1g_1s_2g_2\dots s_ng_n$, for some $n\geq 1$, with $s_j\in P_a$ and $g, g_j$  in a generating set for $P^*$ (which we assume contains the identity) for each $j=1,\dots,n$, see \secref{sec:background}.
We will prove the claim by induction on the number of atoms in the expression of $p$.

Let first $p = gs_1g_1$ where $s_1 \in P_a$ and $g,g_1\in P^*$. Then for  $a\in A_+$ we have
\[\Tr_{\tau}^p(a)=\Tr_{\tau}^{s_1}(a)\leq N_0(s_1)\tau(a)=N_0(p)\tau(a)
\] by \lemref{lem:induced traces constant on principal ideals} and our hypothesis

 Assume now that $\Tr_{\tau}^p\vert_A \leq N_0(p)\tau$ for all $p \in P$ that can be expressed as a product of generators containing $n$ atoms, for $n\geq 2$. Let $q \in P$ be of the form $q=gs_1g_1s_2g_2\dots s_ng_ns_{n+1}g_{n+1}$, with $s_j\in P_a$ and $g,g_j$ in a generating set for $P^*$
 for $j =1, \dots, n+1$. Set $p=g_1s_2g_2\dots s_ng_ns_{n+1}g_{n+1}$, so that $q=gs_1p$.

By the inductive hypothesis applied to $p$, $\tau_0=\Tr_{\tau}^p\vert_A$ is a finite trace on $A$. Then, applying
\lemref{lem:induced traces constant on principal ideals}, equation \eqref{eq:induced trace}, the inductive hypothesis on $p$ and our hypothesis on $s_1$,  we get that
\begin{align*}
\Tr_{\tau}^q(a) = \Tr_{\tau}^{gs_1p}(a) &=\Tr_{\tau_0}^{s_1}(a)= \sup_F \sum_{\xi \in F} \tau_0(\la \xi,a\xi \ra) \\
&\leq \sup_F \sum_{\xi \in F} N_0(p) \tau ( \la \xi, a \xi \ra) \\
&=N_0(p) \Tr_{\tau}^{s_1}(a) \leq N_0(p)N_0(s_1) \tau (a)\\
&=N_0(q)(a)
\end{align*}
for all $a \in A_+$, where $F$ are finite subsets of $X_{s_1}$. We can conclude  that $\Tr_{\tau}^q(a) \leq N_0(q) \tau(a) < \infty$ for all
$a \in A_+$ and all $q \in P$. This shows (i).\\
For (ii), fix a finite subset $U \subset P$ with $\lor U < \infty$. We have that $\tau_U=\Tr_{\tau}^{\vee U}$ is a finite trace on $A$ by assumption. Then
\begin{align*}
\Tr_{\tau_U}^p(a)  = \Tr_{\tau}^{p(\vee U)}(a) < \infty
\end{align*}
for all $a \in A_+$ and $p \in P$ follows by \cite[Proposition 1.2]{LN04}, see equation \eqref{eq:Tr-decomp}.
\end{proof}
We shall need generalisations of \cite[Theorem 2.1]{ALN18} and \cite[Theorem 5.1]{ALN18} valid for right LCM monoids.

\begin{thm}[]\label{thm:ALN 2.1 weak qlo}\textup{(cf. \cite[Theorem 2.1]{ALN18})}
Assume that $P$ is a right LCM monoid and $X=\{X_p\}_{p \in P}$ is a compactly aligned
product system of essential $C^*$-correspondences over $P$ with $X_e=A$. Consider the time evolution on $\Nica \T (X)$
defined by a homomorphism $N: P \to (0,\infty)$, and assume $\phi$ is a
 $\KMS_{\beta}$-state on $\Nica \T (X)$ for some $\beta \in \R$. Then
$\tau = \phi|_{A}$ is a tracial state on $A$ satisfying condition \eqref{eqn:2.1}
for all  finite $J \subset P\bs \{e\}$ and $a\in A_+$.
\end{thm}

\begin{proof}
The argument follows the one from \cite{ALN18}. Let $(\pi,H,\xi)$ be the Hilbert space of the GNS representation of $\NT(X)$ from  $\phi$.
 There are two main points where we need an extension to the right LCM case. First,  we need to know that
\[
N(pu)^{-\beta}\operatorname{Tr}_\tau^{X_{pu}}(\varphi_{pu}(a))=N(p)^{-\beta}\operatorname{Tr}_\tau^{X_p}(\varphi_p(a))
\]
for all $p\in P, u\in P^*$ and $a\in A$. Since $N(u)=1$ because $N$ is a homomorphism, this claim follows directly from \lemref{lem:induced traces constant on principal ideals}.

Second,  we need an analogue of the projections $f_p$, $p\in P$ from \cite{ALN18}. For each $p\in P$, let $1_{pP}$ denote the characteristic function  of the set
$\{q\in P:p\leq q\}$. Then  $B_P=\overline{\operatorname{span}}\{1_{pP}: p\in P\}$ is a commutative $C^*$-subalgebra of
$l^\infty(P)$. Furthermore, just as in the case of quasi-lattice ordered pairs cf \cite[Section 2]{ALN18}, the set $\{1_{pP}\mid p\in P\}$ is linearly independent.

Since the representation $\pi:\mathcal{NT}(X) \to B(H)$  corresponds to a Nica covariant representation $\psi:X\to B(H)$ through $\psi=\pi\circ i_X$, we let $\alpha_p^{\psi}(1)$ be the projection in $\psi_e(A)'\subset B(H)$ from \cite[Proposition 4.1]{F02} (which is  the case of quasi-lattice orders, see \cite[Lemma 2.28]{KL2}  for the right LCM case). We have that $\alpha_p^{\psi}(1)\alpha_q^{\psi}(1)=\alpha_r^{\psi}(1)$ whenever $pP\cap qP=rP$ and $\alpha_p^{\psi}(1)\alpha_q^{\psi}(1)=0$ when $p$ and $q$ do not admit a common right multiple, see  \cite[Proposition 9.5 and Lemma 9.3]{KL1} in connection with \cite[Lemma 2.4]{KL2}. By \cite[Proposition 2.30]{KL2}, there is a representation $L_\pi$ of $B_P$ on $H$ determined by $L_\pi(1_{pP})=\alpha_p^{\psi}(1)$.  In particular,  $\alpha_p^{\psi}(1)$ is the strong limit of $\pi(i_X^{(p)}(u_j))$ for every approximate unit $(u_j)_j$ in $\mathcal{K}(X_p)$. Thus with projections $\alpha_p^{\psi}(1)$ instead of $f_p$, for $p\in P$,  the proof of \cite[Theorem 2.1]{AaHR18} can be adapted verbatim to the case when $P$ is right LCM.
\end{proof}

\subsection{Gauge-invariant KMS states}
It is natural to ask if the necessary condition \eqref{eqn:2.1} describing a KMS state is also sufficient for its existence. A positive answer is given in \cite[Theorem 5.1]{ALN18}, as we now briefly recall. Suppose that $X$ is a compactly aligned product system $X$ over a monoid $P$ in a quasi-lattice ordered pair $(G, P)$. Then there exists a canonical coaction $\delta^G$ of $G$ on $\NT(X)$ that yields a grading into a family of subspaces $\NT(X)_g$ for $g\in G$, cf. \cite[Proposition 3.5]{CLSV11}. A state $\phi$ on $\NT(X)$ is gauge-invariant if it vanishes on all $\NT(X)_g$ where $g\neq e$, in other words if $\phi$ factors through the conditional expectation arising from $\delta^G$, which maps $\NT(X)$ onto its core subalgebra. Then  the map $\phi\mapsto \phi\vert_A$ defines a one-to-one correspondence between the gauge-invariant
$\KMS_\beta$-states on $\Nica \T(X)$ and the tracial states on $A$ satisfying \eqref{eqn:2.1}, cf. \cite[Theorem 5.1]{ALN18}.

In order to extend this result to the case of product systems over right LCM monoids, the immediate question arises as to what should gauge-invariance mean in the absence of an ambient group containing the given monoid. The answer we suggest is that the state should factor through a canonical conditional expectation $E$ of $\NT(X)$, similar to the case of a quasi-lattice ordered  pair $(G, P)$. The reason for this is that the Fock module $\bigoplus_{p\in P} X_p$ induces a grading of $P$ on the reduced Nica-Toeplitz algebra $\NT^r(X)$, and this grading alone suffices to give rise to the desired conditional expectation on $\NT(X)$. We next recall this construction, see \cite{KL1, KL2}.

 Given a compactly aligned product system of (essential) $C^*$-correspondences over a  right LCM monoid $P$, there is a natural Nica covariant representation $l$ of $X$ on
$\bigoplus_{p\in P} X_p$, the Fock representation. For the sake of clarity we use the symbol $l$ for this representation, just as in \cite{F02} (but unlike e.g. \cite{KL2, ALN18}), and reserve $\ell$ for a length function on a monoid. The $C^*$-algebra generated by $l$ is the reduced Nica-Toeplitz algebra $\NT^r(X)$,  and  there is  a canonical surjective $*$-homomorphism $\Lambda:\NT(X)\to \NT^r(X)$, see \cite[Section 2]{KL2}.

The next result is implicit in \cite{KL1},  where it is formulated using the $C^*$-precategory picture of $\NT^r(X)$ and $\NT(X)$, see \cite[Proposition 5.4]{KL1}, respectively \cite[Corollary 6.5]{KL1}. For the sake of completeness and easy reference we state it in a form suitable to the more familiar picture of $\NT^r(X)$  viewed as the closure of the subalgebra spanned by monomials $l(\xi)l(\eta)^*$ for  $\xi\in X_p, \eta\in X_q, p,q\in P$ and similarly with
$\NT(X)$ as the closure of the subalgebra spanned by monomials $i_X(\xi)i_X(\eta)^*$. Denote by  $d:X\to P$  the degree homomorphism of monoids so that the fibre over  $p\in P$ is $X_p=d^{-1}(\{p\})$. By definition, the \emph{core subalgebra} of $\NT(X)$ is the fibre over the identity, namely
\[
B_e^{i_X}=\overline{\operatorname{span}}\{i_X(\xi)i_X(\eta)^*\mid \xi,\eta\in X, d(\xi)=d(\eta)\}
\]
and $B_e^l$ of $\NT^r(X)$ as
\[
B_e^l=\overline{\operatorname{span}}\{l(\xi)l(\eta)^*\mid \xi,\eta\in X, d(\xi)=d(\eta)\}.
\]

\begin{prop} Let $X$ be a compactly aligned product system of essential $C^*$-correspondences over a cancellative right LCM monoid $P$, where $X_e=A$. Then
there is a
faithful conditional expectation $E^r:\NT^r(X)\to B_e^l$ such that
\begin{equation}\label{eq:def expectation reduced NT algebra}
 E^r(\sum_{\xi,\eta\in F}l(\xi)l(\eta)^*)= \sum_{\{\xi\in F\mid d(\xi)=d(\eta)\}}l(\xi)l(\eta)^*,
\end{equation}
where $F$ is a finite subset of $X=\bigcup_p X_p$.

Further, there is a conditional expectation $E:\NT(X)\to B_e^{i_X}$ such that
\begin{equation}\label{eq:def expectation full NT algebra}
E(\sum_{\xi,\eta\in F} i_X(\xi)i_X(\eta)^*)=\sum_{\{\xi\in F\mid d(\xi)=d(\eta)\}}i_X(\xi)i_X(\eta)^*,
\end{equation}
where $F$ is a finite subset of $X=\bigcup_p X_p$.
\end{prop}

\begin{proof}
  A routine calculation shows that the Fock representation $l$ of the product system $X$ and the Fock representation $\mathbb{L}$ of the $C^*$-precategory associated to $X$ are related in the form
\[
\mathbb{L}(\Theta_{\xi,\eta})=l(\xi)l(\eta)^*, \quad \xi\in X_p, \quad \eta\in X_q, \quad p,q\in P,
\]
where $\Theta_{\xi,\eta}\in \mathcal{K}(X_q,X_p)$, as prescribed by \cite[Lemma 2.6]{KL2}. Inserting this in \cite[Equation (5.4)]{KL1} gives the claim about $E^r$. Using the formula for $E^r$ in the proof of  \cite[Corollary 6.5]{KL1} gives the claim about $E$.
\end{proof}

\begin{thm}\label{thm:sufficient gauge inv}\textup{(cf. \cite[Theorem 5.1]{ALN18})}
Assume that $X=\{X_p\}_{p \in P}$ is a compactly aligned product system of essential $C^*$-correspondences over a cancellative right LCM monoid $P$, with $X_e=A$, and consider the time evolution on  $\NT(X)$ defined by a homomorphism $N: P \to (0,\infty)$. For every $\beta \in \R$, the map $\phi\mapsto \phi\vert_A$ defines a one-to-one correspondence between
$\KMS_\beta$-states on $\Nica \T(X)$ that factor through the conditional expectation $E$ and the tracial states on $A$ satisfying \eqref{eqn:2.1}.
\end{thm}

The proof follows the strategy laid down in \cite[Section 5]{ALN18}, which in itself follows by now classical ideas, see for example \cite{OP}.
\begin{lem}\label{lem:gauge invariant}
Assume the hypotheses of \thmref{thm:sufficient gauge inv}. Let $\tau$ be a tracial state on $A$ satisfying \eqref{eqn:2.1}. Then there is a unique state $\phi_0$ on $B_e^{i_X}$ such that
\begin{equation}\label{eq: construct state gauge invariant}
\phi_0(i_X(\xi)i_X(\eta)^*)=N(p)^{-\beta}\tau(\langle\eta, \xi \rangle) \text{ for }\xi,\eta\in X_p, p\in P.
\end{equation}
\end{lem}

\begin{proof}
To get hold of $\phi_0$ we exploit the structure of $B_e^{i_X}$ as the norm closure of a nested family of subalgebras $B_J$ similar to \cite[Section 4]{ALN18}, see also \cite[Lemma 3.6]{CLSV11}. More precisely, we recall that a quasi-lattice $I$ is a partially ordered set such that for all $p,q\in I$, either $p,q$ have a unique least common upper bound, denoted $p\vee q$, or have no common upper bound. The example from \cite[Section 4]{ALN18} is that of a monoid in a quasi-lattice ordered pair. The important observation in \cite[Lemma 4.2]{ALN18} is that the order structure alone, regardless of the monoid structure, allows to express the core $C^*$-subalgebra as $\overline{\bigcup_F B_F}$, where $F$ run over the finite $\vee$-closed subsets of $I$ and $B_F=\oplus_{p\in F} B_p$, with $B_p=i_X^{(p)}(\mathcal{K}(X_p))$.

  In our situation, for a right LCM monoid $P$ we define an equivalence relation  by $x\sim y$ if $y =xu$ for some invertible $u\in P^*$ and we denote the equivalence class of $p$ by $[p] = pP^*$. The crux of the matter is that the quasi-lattice graded structure of  the core and its dense subalgebra $\bigcup_F B_F$ exploited in \cite[Lemma 4.2]{ALN18} is inherited by the set $P/_\sim := \{[p]\mid p\in P\}$ of equivalence classes. Indeed, as noticed in \cite[Section 6]{KL1}, the set $P/_\sim$ is a \emph{quasi-lattice} in the sense of \cite[Section 4]{ALN18} for any right LCM monoid $P$. For each finite $\vee$-closed subset $\tilde{F}$ of $P/_\sim$ we obtain a $C^*$-subalgebra
\[
B_{\tilde{F}}=\overline{\operatorname{span}}\{i_X^{(p)}(\mathcal{K}(X_p))\mid [p]\in \tilde{F}\}
\]
of $\NT(X)$, see \cite[Lemma 6.6]{KL1}. The collection of finite $\vee$-closed\footnote{Note that the definition of $\vee$-closed subset $F\subset P/_\sim$ in \cite[Section 6.6]{KL1} ought to be that $[p]\vee [q]\in F$ whenever $[p],[q]\in F$ admit a common upper bound.} subsets $\tilde{F}$ of $P/_\sim$ forms a directed set, which we denote $\mathcal{F}$, and we have that
\[
B_e^{i_X}=\overline{\bigcup_{\tilde{F}\in \mathcal{F}} B_{\tilde{F}}}
\]
by combining observations in the proof of Corollary 6.3 and of Theorem 6.1(b) from \cite{KL1}.

  Thus we may carry out the construction of a positive linear functional $\phi_0$ given by $\oplus_{[p]\in P/_\sim}\phi_{[p]}$ as in \cite[Proposition 4.4 and section 5]{ALN18}.
\end{proof}

\begin{proof}(Proof of \thmref{thm:sufficient gauge inv}). Let $\phi_0$ be the state on $B_e^{i_X}$ constructed in \lemref{lem:gauge invariant} from a tracial state $\tau$ on $A$ satisfying condition \eqref{eqn:2.1}.  We claim that $\phi:=\phi_0\circ E$ satisfies the KMS condition, in which case we obtain surjectivity of the mapping in the statement of the theorem. To show the KMS$_\beta$ condition, we claim that the calculations in the proof of \cite[Theorem 5.1]{ALN18} are valid using the formula for the expectation $E$ on $\NT(X)$ that comes from the Fock space grading, without reference to an ambient group. Thus, with $a=i_X(\xi)$, $\xi\in X_p$ and $b=i_X(\zeta)i_X(\eta)^*$, for $\zeta\in X_q, \eta\in X_r$, $p,q,r\in P$, we must show that $\phi(ab)=N(p)^{-\beta}\phi(ba)$. The case $pq=r$ is verbatim as in \cite{ALN18}, but the case $pq\neq r$ requires some attention as we cannot form elements $pq^{-1}r$ and $qr^{-1}p$ in the absence of a group.

 Assume that $pq\neq r$. Since $ab=i_X(\xi\zeta)i_X(\eta)^*$, it follows from \eqref{eq:def expectation full NT algebra} that $\phi(ab)=0$. Turning to $ba$, we rewrite this following the idea of the proof of \cite[Proposition 2.10]{KL1}, thus we let $\eta=R\eta'$ and $\xi=P\xi'$  for $R\in \mathcal{K}(X_r)$, $P\in \mathcal{K}(X_p)$, $\eta'\in X_r, \xi'\in X_p$. Then
 \[
 ba=i_X(\zeta)i_X(\eta')^* i_X^{(r)}(R^*)i_X^{(p)}(P)i_X(\xi').
 \]
 By Nica covariance of $i_X$ we have that $ba=0$ if $rP\cap pP= \emptyset$. Assume therefore that $rP\cap pP\neq \emptyset$. We have $rP\cap pP=sP$ with $s\in P$, and we let $pp'=rr'=s$ for unique $p',r'$ in $P$. Therefore $ba\in i_X(X_{qr'})i_X(X_{p'})^*$, and we must prove that
 %when $pP\cap rP\neq \emptyset$, we have $pq\neq r$  precisely when
 $qr'\neq p'$, in which case we can conclude that $\phi(ba)=0$ by the definition of $E$. If $qr'=p'$, then $pqr'=pp'=rr'$, and right cancellation in $P$ would give $pq=r$, a contradiction.

 Injectivity of the mapping $\phi\mapsto \phi\vert_{A}$ follows because the restriction to $B_e^{i_X}$ of a KMS state is determined uniquely by the requirement \eqref{eq: construct state gauge invariant}.
\end{proof}

%%%%%%%%%%%%%%

\section{Reduction of the positivity condition to atoms}\label{sec:tree}
Throughout this section we assume that $P$ is a Noetherian  left cancellative right LCM monoid, that
$\{ X_p\}_{p \in P}$ is a compactly aligned product system of $C^*$-correspondences
over $P$ with $X_e = A$, that
$N:P \to (0,\infty)$ is a multiplicative homomorphism and that $\beta \in \R$. We also assume that $\tau$ is a tracial state of $A$ satisfying
the positivity condition for finite subsets of atoms:
\begin{equation}\label{eqn:atomicpositivity}
\tau(a) +
\sum_{\emptyset \neq U \subset J}
(-1)^{|U|}N(\lor U)^{-\beta}\Tr_{\tau}^{ \lor U}(a) \geq 0 \qquad \text{ for } a \in A_+\quad J\subset P_a.
\end{equation}
The goal of the section is to show that, under an extra hypothesis on $P$,
the positivity also holds for all finite subsets of $P\bs \{e\}$, as in \eqref{eqn:2.1}. This gives a version of \thmref{thm:sufficient gauge inv} that is stronger because it only requires verification of the  subcollection of inequalities arising from subsets of atoms, which  is finite if $P$ is finitely generated.
We are motivated by the simplification obtained for right-angled Artin monoids in \cite[Theorem 9.1]{ALN18}.

The key property we require for the reduction to atoms is the finiteness of a tree that we construct, largely following  the proof of \cite[Lemma 4.2]{BL18}. In the formal arguments we replace finite subsets $J$ of $P$ by finite \emph{lists} of elements of $P$, in which repetitions are allowed. The reason is that even if the initial list has no repetitions, repetitions may appear already at the second step.

\subsection{Positivity for finite lists in $P \cup\{\infty\}$}
We will use sums over lists of elements in $P$
instead of subsets of $P$., and it is also convenient to allow $\infty$ in our lists.  So for each  $n \in \N$ we consider the index set $I_n\coloneqq \{1, 2, \ldots, n\}$ and define  an $n$-list to be a function of the form $\lambda: I_n \to P \cup \{\infty\}$.

\begin{lem}
Given a tracial state $\tau$ on $A$ that satisfies \eqref{eqn:atomicpositivity}, a list $\lambda$  in $P \cup\{\infty\}$, and $\beta>0$, there is a tracial linear functional $a\mapsto Z(\lambda,\tau,a)$ defined for each $a\in A$ such that
\begin{equation}
Z(\lambda,\tau,a) := \sum_{ U \subset I_n} (-1)^{|U|} N(\lor \lambda(U))^{-\beta} \Tr_{\tau}^{\lor \lambda(U)}(a), \quad a\in A,
\end{equation}
where we follow the convention that if $\lambda(U)$ is not a clique, then $\vee \lambda(U) = \infty$ and $N(\lor \lambda(U))^{-\beta} = 0$.
\end{lem}
\begin{proof}
When the subset $J$ consists of a single atom, \eqref{eqn:atomicpositivity} is  the assumption of  \lemref{lem:Tr-finite-onA-from-S}(i), and then  \lemref{lem:Tr-finite-onA-from-S}(ii) implies that  $\Tr_{\tau}^{\lor \lambda(U)}$ is a tracial  linear functional defined on all of $A$ whenever $\lor \lambda(U) \in P$.
When $\lor \lambda(U) = \infty$ it is irrelevant that $\Tr_{\tau}^{\lor \lambda(U)}(a)$ is not defined because then the corresponding coefficient vanishes. Hence $Z(\lambda,\tau,a)$ is well-defined and the function $a\mapsto  Z(\lambda,\tau,a)$ is a tracial bounded linear functional on $A$.
\end{proof}
For each list $\lambda$, we let   $\clq_\lambda \coloneqq \{ U \subset I_n: \lambda(U) \in \clq(P)\}$ be the collection of cliques in $\lambda$, explicitly,  this is the collection of subsets of $I_n$ on which $\lambda$ has an upper bound in $P$. Since the terms associated to non-cliques vanish, we have
 \begin{equation*}
Z(\lambda,\tau,a) %&:= \sum_{ U \subset I_n} (-1)^{|U|} N(\lor \lambda(U))^{-\beta} \Tr_{\tau}^{\lor \lambda(U)}(a)
= \sum_{ U \in \clq_\lambda} (-1)^{|U|} N(\lor \lambda(U))^{-\beta} \Tr_{\tau}^{\lor \lambda(U)}(a),
\end{equation*}
where the sums are over over cliques  or over lists, indistinctly.

If $J\subset P$ is a set with $n$ elements, we replace it by a list $\lambda: I_n \to P$ with $\lambda(I_n) = J$ and write $Z(\lambda,\tau,a)$  as a sum over subsets of $P$.
Again, the sum only depends on the part of the range of the list that lies in $P$, and  terms corresponding to non-cliques vanish, that is,
\[
Z(\lambda,\tau,a)  =\sum_{ K \subset \lambda(I)\cap P} (-1)^{|K|} N(\lor K)^{-\beta} \Tr_{\tau}^{\lor K}(a)=\sum_{ K \in \clq(\lambda(I))} (-1)^{|K|} N(\lor K)^{-\beta} \Tr_{\tau}^{\lor K}(a).
\]
We see next that multiples, repetitions, and invertible elements in a list can be ignored as far as $Z(\lambda,\tau,a)$ is concerned.

\begin{lem}
\label{lem:eliminatingmultiples}  Suppose $\lambda :I_n \to P$
is a list and assume $\lambda(i) \leq \lambda(j)$ for some $i$ and $j$ with $i\neq j$. Denote by
$\hat\lambda$ the list obtained by restricting  $\lambda$ to $I_n\setminus \{j\}$. Then $Z(\lambda,\tau,a) = Z(\hat\lambda ,\tau,a)$.
In particular if $\lambda(I_n) \cap P^* \neq \emptyset$, then $Z(\lambda,\tau,a) = 0$ for all $a$.
 \begin{proof}
Suppose $\lambda(i) \leq \lambda(j)$  for $i,j \in I_n$ with $i \neq j$. The $\lambda$-cliques $U\subset I_n$  that contain $j$ fall into two classes, namely those that contain $i$ and those that do not. The operation of `removing $i$' establishes a bijection of the first class onto the second one, reducing the cardinality of the clique by one but leaving the least upper bound unchanged.  This changes the sign of the corresponding summand
of $Z(\lambda,\tau,a)$, but not its absolute value, so  summands paired this way cancel each other out.

If we assume now $\lambda(i)\in P^*$, then $\lambda(i) \leq \lambda(j)$ for every $j\in I_n$, and we may remove successively each $\lambda(j) $ for $j\neq i$ from $\lambda$, finally arriving at $Z(\lambda,\tau,a) = Z(\{i\},\tau, a)$ which is easily seen to be zero because the only two cliques are $\emptyset$ and $\{i\}$, which have opposite signs and the same absolute value $\tau(a)$.
 \end{proof}
 \end{lem}

\begin{defn}\label{def:leaf}
A list $\lambda: I_n \to P\cup\{\infty\}$ will be called a \emph{leaf} if  either $\lambda(I_n) \subset P_a  \cup\{\infty\}$ or $\lambda(I_n) \cap P^* \neq \emptyset$.
 \end{defn}

 Next we define an iteration step.  Let $\lambda$ be a list in $P\cup \{\infty\}$. When $\lambda$ is not a leaf, let $i$ be the smallest number in $\{1, 2, \ldots, n\}$ for which $\lambda(i) = pq$ with $p,q\in P\setminus  P^*$ (here we may assume without loss of generality that $p$ is a generator) and define two new lists, $\lambda_1$ and $\lambda_2$, in $P \cup \{\infty\}$ as follows.

 \begin{equation}\label{eqn:recursivestep}
 \lambda_1(j) \coloneqq \begin{cases}\lambda(j) & \text{ if } j\neq i\\
 p & \text{ if } j = i;
 \end{cases}
\qquad \qquad\qquad
\lambda_2(j) \coloneqq \begin{cases}p\inv (p \vee \lambda(j)) & \text{ if } p \vee \lambda(j) <\infty \\
 \infty & \text{ if } p \vee \lambda(j)  =\infty.
 \end{cases}
 \end{equation}

In English: the list  $\lambda_1$ is obtained from $\lambda$ by simply replacing the $i^{th}$ term $pq$  by $p$
and leaving the rest of the terms unchanged,
and the  list $\lambda_2$ is obtained by replacing every term in $\lambda$ by its image
under the transformation $x \mapsto p\inv (x\vee p)$ (here we let $p\vee x$ be any choice of least common upper bound of  $p$ and $x$  if there is one and $\infty$ if there is none).
Notice that, in particular, we may choose
$\lambda_2(i) =p \inv (p\vee pq) =q$.

Clearly  $\lambda_1 (j) =\infty$ if and only if $\lambda(j) =\infty$, and   $\lambda_2(j) = \infty$ if $\lambda(j) =\infty$, but it is certainly possible to have   $\lambda_2(j) =\infty$ with $\lambda (j) < \infty$.

\begin{lem}\label{lem:lambda1lambda2arecliques}
Let $P$ be a right LCM monoid. Let  $\lambda:I_n \to P\cup \{ \infty\}$ be a list that is not a leaf and let $i \in I$ be the first index such that $\lambda(i) = pq$ for  $p\in P_a$ and $q\in P\setminus P^*$.  For each $U \subset I_n$ we have
\begin{enumerate}
\item if $\lambda(U)$ is a clique, then so is $\lambda_1(U)$, and  \[\bigvee \lambda_1(U) \leq \bigvee \lambda(U);\]

\item $\lambda(U) \cup \{p\}$ is a clique if and only if $\lambda_2 (U)$ is a clique, and
\[
\bigvee \lambda(U) \vee p = p \cdot \bigvee \lambda_2(U); \]

\item when $i\in U$ we have that $\lambda_1(U)$ is a clique if and only if $\lambda_2(U\setminus\{i\})$ is a clique, and
\[  \bigvee \lambda_1(U) = p\cdot \bigvee \lambda_2(U\setminus\{i\}).
\]
\end{enumerate}

\end{lem}

\begin{proof}  Suppose $\lambda$ is not a leaf and $\lambda(i) =pq$ as in the statement.
Part $(1)$ can be split into two cases. First assume $i  \notin U$; then obviously $\lambda(U) = \lambda_1(U)$, and the claim follows.
Next assume $i \in U$; then
clearly $\lambda(i) P = pqP \subset pP = \lambda_1(i)P$, and thus $\bigcap _{j\in U} \lambda(j) P \subset \bigcap_{j\in U} \lambda_1(j)P$.  This gives the
inequality in part (1).  Cliques are preserved because $\bigvee \lambda(U) < \infty$ implies
$\bigcap_j \lambda_1(j) P\neq \emptyset$.

For part (2), notice that
\begin{equation}\label{eqn:cliquesintersectioncase2}
p \bigcap_{j\in U} \lambda_2(j)P= \bigcap_{j\in U} p\lambda_2(j)P = \bigcap_{j\in U}   (p\vee \lambda(j)) P =  \bigcap_{j\in U} (pP \cap \lambda(j)P), %=\bigcap_{j\in U} \lambda(j)P,
\end{equation}
 Interpreting this for
sets of right LCMs, we see that
\[
p \cdot \bigvee \lambda_2(U) = \bigvee p\lambda_2(U) = p \vee \bigvee \lambda(U) = \bigvee (\lambda(U) \cup \{p\}),
\]
which proves the second assertion.

For part (3), we notice that since $\lambda_1(j) = \lambda(j)$ for $j\in U\setminus\{i\}$, the intersection over $U\bs \{i\}$ can be transformed into one over $U$; specifically,
\begin{equation}\label{eqn:cliquesintersectioncase3}
p \bigcap_{j\in U\setminus\{i\}} \lambda_2(j)P= (\bigcap_{j\in U\setminus\{i\}} p\lambda_2(j)P)  \cap pP= \bigcap_{j\in U\setminus\{i\}} (pP \cap \lambda(j)P) \cap \lambda_1(i) P=\bigcap_{j\in U} \lambda_1(j)P,
\end{equation}
so the sets of right LCMs satisfy
\[
p \cdot \bigvee \lambda_2(U \bs \{i\}) = \bigvee \lambda_1(U).
\]
That cliques are preserved both ways in cases (2) and (3) is easy to see from  \eqref{eqn:cliquesintersectioncase2} and \eqref{eqn:cliquesintersectioncase3}.
\end{proof}

\begin{rem} In some cases the map $\lambda \mapsto \lambda_1$ can send an unbounded list to a clique. For an easy example  consider
 $P=\FF_2^+$, the free monoid on two generators $a$ and $b$. Then the list $\lambda=(ab,a^2)$  is unbounded, but $\lambda_1= (a,a^2)$ is a clique.

 The map $\lambda \mapsto \lambda_1$ does not introduce new invertibles  or new infinites. On the other hand, $\lambda \mapsto \lambda_2$ has the capacity to generate new invertibles and new infinites. %One can construct easy examples using $\FF_2^+$.
\end{rem}
\begin{lem}\label{lem:trace}
Let $X$ be an essential $C^*$-correspondence over $A$.
Suppose $\tau_1, \dots, \tau_n$ are positive finite traces on $A$ such that $\Tr_{\tau_i}^X$ is finite for each $i = 1, \dots, n$, and let  $c_1, \dots, c_n \in \R$. If $\sum_{i = 1}^n c_i \tau_i$ is positive on $A$, then so is $\sum_{i=1}^n c_i \Tr_{\tau_i}^X(\cdot)$.
\end{lem}

\begin{proof}
By the linearity properties of the induced trace
we have
\begin{align*}
 \sum_{i = 1}^n c_i \Tr_{\tau_i}^X(a)  =
\Tr_{\sum_{i = 1}^n c_i \tau_i}^X(a) \quad \text{ for all }a\in A_+.
\end{align*}
The result follows because  inducing preserves positivity.
\end{proof}

\begin{prop}\label{prop:recursion lists semilattice}  Suppose
$P$ is a Noetherian  left cancellative right LCM monoid, that
$\{ X_p\}_{p \in P}$ is a compactly aligned product system of $C^*$-correspondences
over $P$ with $X_e = A$, that
$N:P \to (0,\infty)$ is a multiplicative homomorphism and that $\beta \in \R$.
Assume that $\tau$ is a tracial state of $A$ satisfying
the positivity condition \eqref{eqn:atomicpositivity} for finite subsets of atoms.

If $\lambda$,  $\lambda_1$, and $\lambda_2$ are lists, with $\lambda(i) = pq$ as in \eqref{eqn:recursivestep}, then
\begin{equation*}\label{eq:recursion formula lists}
Z(\lambda,\tau,a)=Z(\lambda_1,\tau,a)+
\frac{1}{N(p)^{\beta}} \sum_{U\subset I_n}(-1)^{|U|} N(\vee \lambda_2(U))^{-\beta}\Tr^p_{F_{\vee \lambda_2(U)}(\tau)}(a)
%N(p)^{-\beta}\Tr^p_{\tilde{\tau}}(a)
\quad \text{ for all }a\in A_+.
\end{equation*}
\end{prop}

\begin{proof}
Let $i\in I_n$ be given as in \eqref{eqn:recursivestep}. We compute the difference $Z(\lambda,\tau,a)-Z(\lambda_1,\tau,a)$ for $a\in A_+$.
Notice that  the terms corresponding to each clique $U\subset I_n$ with $i\notin U$ cancel, because then $\lambda(U)=\lambda_1(U)$.  So we only need to write terms with $i\in U \subset I_n$. Therefore
\begin{equation*}
Z(\lambda,\tau,a)-Z(\lambda_1,\tau,a)
=\sum_{i\in U\subset I_n}(-1)^{|U|}\biggl(N(\vee \lambda(U))^{-\beta}\Tr_\tau^{\vee \lambda(U)}(a)- N(\vee \lambda_1(U))^{-\beta}\Tr_\tau^{\vee \lambda_1(U)}(a)\biggr)
\end{equation*}
Since $i\in U$, so that $\lambda(U) \cup \{p\} = \lambda(U)$,  \lemref{lem:lambda1lambda2arecliques}(2) gives $N(\vee \lambda(U)) = N(p) N(\vee \lambda_2(U))$ and
 \lemref{lem:lambda1lambda2arecliques}(3)  gives $  N(\vee \lambda_1(U))=N(p) N(\vee \lambda_2(U\setminus\{i\}))$, so
\begin{equation*}
Z(\lambda,\tau,a)-Z(\lambda_1,\tau,a)
=\sum_{i\in U\subset I_n}\frac{(-1)^{|U|}}{N(p)^{\beta} }\biggl(N(\vee \lambda_2(U))^{-\beta}\Tr_\tau^{\vee \lambda(U)}(a) -
N(\vee \lambda_2(U\setminus\{i\}))^{-\beta}\Tr_\tau^{\vee \lambda_1(U)}(a)\biggr).
\end{equation*}
When we replace  $U\setminus \{i\}$ by $V$, which is a set that ranges over the subsets of $I_n \setminus \{i\}$, and we adjust the sign to take into account that $|V| =|U| -1$, we get
\begin{align}\label{eqn:formulaZ-Z1}
Z(\lambda,&\tau,a)-Z(\lambda_1,\tau,a)
=\frac{1}{N(p)^{\beta}}\biggl( \sum_{i\in U\subset I_n}(-1)^{|U|} N(\vee \lambda_2(U))^{-\beta}\Tr_\tau^{ \bigvee \lambda(U)}(a)\\
&+\sum_{i\notin V\subset I_n}(-1)^{|V|} N(\vee\lambda_2(V))^{-\beta}\Tr_\tau^{\lambda_1(i) \vee\bigvee\lambda_1(V)}(a)\biggr).\notag
\end{align}
Recall now that by \eqref{eq:Tr-decomp} we have $
\Tr_\tau^{ \bigvee \lambda(U)}=\Tr^p_{F_{\vee\lambda_2(U)}(\tau)}$ for each $U$ containing $i$ and
$\Tr_\tau^{\lambda_1(i)\vee \bigvee \lambda(V)}=\Tr^p_{F_{\vee \lambda_2(V)}(\tau)}$ for each $V$ that does not contain $i$. Thus we may collect the two sums from \eqref{eqn:formulaZ-Z1} into the single sum, completing the proof.
\end{proof}

By recursion, the maps $\lambda \mapsto \lambda_1$ and $\lambda \mapsto \lambda_2$ generate a binary tree rooted at $\lambda$. This is  similar to the proof of \cite[Theorem 4.8]{BL18}, but here we use the following tree with nodes  indexed by lists,
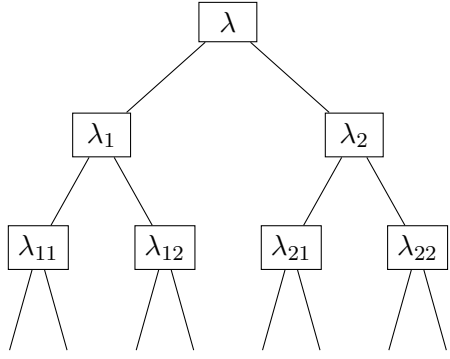
\tikzset{every tree node/.style={minimum width=2em,draw},
         blank/.style={draw=none},
         edge from parent/.style=
         {draw,edge from parent path={(\tikzparentnode) -- (\tikzchildnode)}},
         level distance=1.5cm}

\begin{figure}[H]
\centering
\begin{tikzpicture}
\Tree
[.$\lambda$
    [.$\lambda_1$
    	\edge[]; [.$\lambda_{11}$
    		\edge[]; \node[blank]{};
    		\edge[]; \node[blank]{};
    		]
    	\edge[]; [.$\lambda_{12}$
    		\edge[]; \node[blank] {};
    		\edge[]; \node[blank] {};
    		]
    ]
    [.$\lambda_2$
    	\edge[]; [.$\lambda_{21}$
    		\edge[]; \node[blank]{};
    		\edge[]; \node[blank]{};
    		]
    	\edge[]; [.$\lambda_{22}$
    		\edge[]; \node[blank]{};
    		\edge[]; \node[blank]{};
    		]
    ]
]
\end{tikzpicture}
\end{figure}
At level $k$ we have nodes indexed by $\lambda_\omega $, where $\omega$ is a word of length $k$ in the free monoid $ \FF^+\{1,2\}$ generated by the symbols $1$ and $2$.
The level $k+1$ is generated by splitting each node $\lambda_{\omega} $ into two new nodes $\lambda_{\omega 1} :=(\lambda_\omega)_1 $ and
$\lambda_{\omega 2} :=(\lambda_\omega)_2$ defined according to the recursive step \eqref{eqn:recursivestep}. As a matter of notation, whenever $\omega$ is a
finite or infinite word with entries in $\{1, 2\}$, we denote the initial subword of length $k$ by $\omega[1,k]$.

According to \defref{def:leaf}  a node $\lambda_\omega $ is a leaf if  $\lambda_\omega(j) \in P^*$ for some $j$, or if for every $j \in I_n$ there is no notrivial factorisation of $\lambda_\omega(j)$. We are interested in a branching process that stops when $Z(\lambda_\omega,\tau,a) \geq 0$, and this will be ensured either by the positivity assumption for lists consisting of atoms and infinities, or by \lemref{lem:eliminatingmultiples} for lists that have an invertible element, because then $Z(\lambda,\tau,a) = 0$.
Thus we will say that a {\em branch} is a finite or infinite word $\omega$ on the symbols $\{1,2\}$
that starts at the root $\lambda $ and either ends at the first node such that $\lambda_{\omega[1,k]} $ is a leaf, or else does not end at all, should such a node not exist.

\begin{thm} \label{thm:finite type}
Under the assumptions of \proref{prop:recursion lists semilattice}, suppose that the tree $(\lambda_\omega)$ constructed above is finite for every list $\lambda$.
Then reduction of positivity to atoms holds for $P$. Specifically, a tracial state $\tau$ on $A$ satisfies \eqref{eqn:2.1} if and only if it  satisfies the weaker condition \eqref{eqn:atomicpositivity}.
\end{thm}

\begin{proof} Suppose that $\tau$ is a tracial state of $A$ that satisfies the positivity condition \eqref{eqn:atomicpositivity} for subsets of atoms. Let $J$ be a subset of $P$ and take a list $\lambda$ with range $J$. The left hand side of \eqref{eqn:2.1} then equals $Z(\lambda, \tau, a)$
so it suffices to show that $Z(\lambda, \tau, a) \geq 0$ for every $a\in A^+$.

 Assume that $\lambda$ is a list in $P$ such that the positivity condition \eqref{eqn:2.1} holds for both $\lambda_1$ and $\lambda_2$ for every $a\in A$.
Notice that
\[
Z(\lambda_2,\tau,\cdot) = \sum_{ U \subset I_n} (-1)^{|U|} N(\lor \lambda_2(U))^{-\beta} \Tr_{\tau}^{\lor \lambda_2(U)}  = \sum_{ U \subset I_n} (-1)^{|U|} N(\lor \lambda_2(U))^{-\beta} (F_{\lor \lambda_2(U)} \tau)
\]
is a linear combination of the finite positive traces $F_{\lor \lambda_2(U)} \tau $. Since $Z(\lambda_2,\tau,\cdot) \geq 0$ by assumption, \lemref{lem:trace} implies that $ \frac{1}{N(p)^{\beta}} \sum_{U\subset I_n}(-1)^{|U|} N(\vee \lambda_2(U))^{-\beta}\Tr^p_{F_{\vee \lambda_2(U)}(\tau)} (a) \geq 0$, and then   \proref{prop:recursion lists semilattice} shows that positivity holds for $\lambda$.

Hence, when positivity fails for a list $\lambda$, then it must also fail for either
$\lambda_1$ or $\lambda_2$, and iteration of this process, say, choosing the first edge whenever  positivity fails for both, generates a branch of the tree of $\lambda$ such that positivity fails at every node. Since the tree of $\lambda$ is finite by assumption, this branch must end at a leaf, on which positivity also fails. But this is a contradiction, because infinities in a list do not affect the sum, invertible elements cause the function $Z$ to vanish, and if the list consists only of atoms, then the value is nonnegative by  assumption \eqref{eqn:atomicpositivity}.
\end{proof}
As an immediate consequence we obtain a stronger version of \thmref{thm:sufficient gauge inv}.
\begin{corollary}\label{cor:simplification}
If in addition to the assumptions of \thmref{thm:sufficient gauge inv} the tree of every list is finite, then $\phi \mapsto \phi|_A$  is a one-to-one correspondence between
$\KMS_\beta$-states on $\Nica \T(X)$ that factor through the conditional expectation $E$ and the tracial states on $A$ satisfying  \eqref{eqn:atomicpositivity}.
\end{corollary}

\subsection{The directed case} If the right LCM  monoid $P$ is directed (hence a semi-lattice), in the sense that $pP \cap qP \neq \emptyset$ for every $p,q\in P$, we can actually prove that the tree $(\lambda_\omega)$ is finite.
\begin{prop}\label{pro:finitetree-directedcase}
Suppose $P$ is a Noetherian, right LCM monoid such that $xP \cap yP \neq \emptyset$ for every pair $x,y \in P$,
and let $\lambda:I_n \to P$ be a list of elements of $P$. Then the tree $(\lambda_\omega)$  is finite, and reduction of positivity to atoms holds for $P$.
\begin{proof}
By K\"{o}nig's lemma,
the tree is infinite if and only if it has an infinite branch.
By definition, an infinite branch consists of an infinite sequence $\lambda_\omega  = (\lambda_{\omega[1,k]} )_{k\in\N}$ of nodes associated to
 an infinite word $\omega$ over $\{1,2\}$ such that for every level $k$ the set $\lambda_{\omega[1,k]}(I_n)$ contains no invertible elements and has at least one element
with a nontrivial factorisation.  We will show that  there are no such infinite branches.

We claim first that a branch $\omega$ can only have finitely many $2$'s.
Aiming for a  contradiction, we assume $\omega$ has infinitely many $2$'s, occurring precisely as
$\omega_{k_j}$ for an infinite increasing sequence $({k_j})_{j\in\N}$, while the rest of the $\omega_i$ are equal to $1$.
Thus,  $\omega[1,k_1-1]$ is a string consisting of $k_1-1\geq 0$ symbols equal to $1$,  followed by $\omega_{k_1} =2$ (the first $2$ in $\omega$) followed by the string $\omega[k_1+1,k_2-1]$ of $k_2-k_1 -1 \geq 0$ symbols equal to $1$, followed
 $\omega_{k_2}$ (the second $2$ in $\omega$), and so on. Since $\omega(j) =1$ for
 $j \leq k_1 -1$. \lemref{lem:lambda1lambda2arecliques}(1) applied $k_1-1$ times gives
\[\vee \lambda =\vee \lambda_{\omega[1]}g'_1=(\vee \lambda_{\omega[1,2]}g'_2)g'_1=\dots = (\vee \lambda_{\omega[1,k_1-1]}g'_{k_1-1})g'_{k_1-2}\dots g'_2g'_1,
\]
and we set $g_{k_1}=g'_{k_1-1}\dots g'_2g'_1$, where $g_{k_1} =e$ in case $k_1 =1$.
Since $\omega_{k_1} =2$, \lemref{lem:lambda1lambda2arecliques}(2) applied to the list $\lambda_{\omega[1,k_1-1]}$
with $U =I_n$  gives
\[
\vee \lambda_{\omega[1,k_1-1]}=p_{k_1} (\vee (\lambda_{\omega[1,k_1-1]})_2) = p_{k_1} (\vee \lambda_{\omega[1,k_1]})
\]
for some generator $p_{k_1}\in P_a$. Hence,  writing $d_k\coloneqq \vee \lambda_{\omega[1,k]}$ for $k = 1, 2, \ldots$, %there is $g_{k_1}\in P$  such that $d_0=d_{k_1}g_{k_1}$.
we have shown that
\begin{equation}\label{eq:sup first level 2}
 d_0= \vee \lambda = d_{k_1-1}g_{k_1}=p_{k_1}d_{k_1} g_{k_1}.
\end{equation}

Now, let $k_2> k_1 $ be the second occurrence of a $2$ in $\omega$.
If $k_2 = k_1 +1$, set $g_{k_2} = e$.
Otherwise,  by \lemref{lem:lambda1lambda2arecliques}(1) applied $k_2-k_1-1$ times  and
by \lemref{lem:lambda1lambda2arecliques} applied at the $k_2$ step
 we find  $g'_j\in P$ for $j=k_1+ 1,\dots, k_2-1$  and a generator $p_{k_2}\in P_a $  such that
\[
\vee \lambda_{\omega[1,k_2-1]}=p_{k_2}(\vee \lambda_{\omega[1,k_2]})
\]
and
\[
\vee \lambda_{\omega[1,k_1]}=\vee \lambda_{\omega[1,k_2-1]}g'_{k_2-1}\dots g'_{k_1+1}.
\]
Putting $g_{k_2}=g'_{k_2-1}\dots g'_{k_1+1}$,   %and multiplying it on the right of $d_{k_2}$
gives
\[
d_{k_1}=d_{k_2-1}g_{k_2}= p_{k_2}d_{k_2}g_{k_2}.
\]
Inserting this in  \eqref{eq:sup first level 2} we get
\begin{equation*}d_0=p_{k_1} (p_{k_2} d_{k_2}g_{k_2}) g_{k_1}.
\end{equation*}
Continuing this way we get two sequences $\{p_{k_m}\}$ in $P\setminus P^*$ and $\{g_{k_m}\}$ in $P$ such that
\[
d_0=p_{k_1}\dots p_{k_m}d_{k_m+1}g_{k_m}\dots g_{k_1}.
\]
Thus $(p_{k_1}, \, p_{k_1}p_{k_2}, \, p_{k_1} p_{k_2}p_{k_3}, \, \ldots )$  is a strictly increasing sequence of left divisors of $d_0$, which contradicts the assumption that $P$ is Noetherian and completes the proof of the claim.

It follows that there exists $\bar k$  such that $\omega[k]=1$ for all $k\geq \bar k$.  We will show that this leads to a contradiction.
Notice that, by definition, the process defining $\lambda_1$ will exhaust the factorisation of  $\lambda_{\omega_{[1,\bar k]}}(1)$ before moving onto that of
$\lambda_{\omega_{[1,\bar k]}}(2)$, and so on. Continuing, we see that the $j$-th term $\lambda_{\omega_{[1,\bar k]}}(j)$ of the list remains constant down the branch until it is changed for the first time after the initial $j-1$ terms have been replaced by generators,  admitting no further nontrivial factorisations.
Each step in the $\lambda_1$-process requires a nontrivial factorisation of exactly one of the terms in the list at that level, that is, at the  $(\bar k+y)$-th step, there exist exactly one $j\in I_n$,  $p_y \in P_a$ and $q_y \in P\setminus P^*$ such that $\lambda_{\omega_{[1,\bar k+y ]}} (j) = p_y q_y$ and
$\lambda_{\omega_{[1,\bar k+y +1]} } = p_y$. The other terms in the list remain unchanged.
%Since $P$ is Noetherian each term $\lambda_{\omega_{[1,\bar k]}}(j)$ can be factorized only finitely many times, and
Since the list is finite, the $\lambda_1$-process must stop after finitely many steps. This contradicts the assumption that $\omega$ is an infinite branch, proving that the tree has to be  finite. The last assertion now follows by \thmref{thm:finite type}
\end{proof}
\end{prop}
The question of which monoids have reduction of positivity to subsets of atoms was raised  in the introduction of \cite{ALN18}.
As a corollary we can add Artin monoids of finite type to that class.
\begin{corollary}\label{cor:directedcase}
Reduction of positivity to generators holds for Artin monoids of finite type.
\end{corollary}
\begin{proof}
The result follows from \proref{pro:finitetree-directedcase} because Artin monoids of finite type are Noetherian, directed, right LCM monoids.
\end{proof}

\subsection{Right-angled Artin monoids} Our methods also recover the fact  that reduction of positivity holds for right-angled Artin monoids, cf. \cite[Theorem 9.1]{ALN18}.

\begin{prop}\label{pro:finitetree-RAAMcase}
Suppose $P$ is a right-angled Artin monoid,
and let $\lambda:I_n \to P$ be a list of elements of $P$. Then the tree $(\lambda_\omega)$  is finite, and reduction of positivity to generators holds for $P$.
\end{prop}

\begin{proof}
Let $\omega \in \mathbb{F}^+\{1,2\}$, and suppose
that $\lambda_{\omega[1,k]}(I_n)$ is not a leaf for some $k \in \N$. Recall from the definition of $\lambda_1$ and $\lambda_2$ in \eqref{eqn:recursivestep}  that when we factor the element $\lambda(i) =sq$ in $\lambda_{\omega}(I_n)$, we choose $s$ to be a generator, that is $s \in P_a$.

Let
$\ell: P \to \N$ be the normalized length function on $P$ (see \Cref{rem:length}). We  claim that $\lambda_{\omega[1,k+1]}(j)$ either satisfies
 $\lambda_{\omega[1,k+1]}(j) = \infty$ or else $\ell(\lambda_{\omega[1,k+1]}(j)) \leq \ell(\lambda_{\omega[1,k]}(j))$ for all $j = 1, \dots, n$, and that the inequality is strict
 at least for the subindex $i$ where the factorisation occurs.
Indeed, the length reduction clearly holds in $\lambda_1$-step, that is, when $\omega_{k+1} =1$, because the length of
the element $\lambda_{\omega[1,k]}(i)$ being factorised goes down to $1$, while the lengths of all the other elements remain the same.
Next we show that the length reduction also holds in the case of a
$\lambda_2$-step, that is, when $\omega_{k+1} = 2$. From  \cite[Lemma 9.2]{ALN18}, we know that $s \vee p \in \{ p, s p, s, \infty\}$ for all $p \in P$ and
$s \in P_a$. Thus $s^{-1}(s \vee p) \in \{s^{-1}p,p,e,\infty \}$, where the case $s^{-1}p$ only occurs if $s$ is an initial letter in some expression of $p$, in which case we have that $s^{-1}p \in P$ with $\ell(s^{-1}p) < \ell(p)$.  From this it becomes clear that either
$\lambda_{\omega[1,k+1]}(j) = \infty$ or
\begin{align*}
\lambda_{\omega[1,k+1]}(i) = \ell(s^{-1}(p \vee s)) \leq \ell(p) = \ell(\lambda_{\omega[1,k]}(i)),
\end{align*}
for all $i = 1, \dots, n$. In this case the term $\lambda_{\omega[1,k]}(i)$ being factorised goes down in length by $1$.
This finishes the proof of our claim.

If we now account for the total length $\sum_j \ell(\lambda(j))$ using the convention that $\ell(e) = \ell(\infty) =0$, we see that
$k \mapsto \sum_j \ell(\lambda_{\omega[1,k]}(j))$ is a strictly decreasing  sequence with values in the positive integers. This excludes the possibility of an infinite branch, and thus the tree
$(\lambda_{\omega})$ is finite by  K\"{o}nig's lemma.
\end{proof}

\section{Reduction to minimal elements}\label{sec:reduction minimal set all qlo}

In this section we show that for all weak quasi-lattice ordered groups $(G, P)$ with $P$ Noetherian there is
a reduction of the positivity criterion  \eqref{eqn:2.1} to a system of inequalities associated to subsets of a set $\pmin$, see the definition below, that contains the atoms but may have other elements, although in some special cases we can verify that
$\pmin = P_a$. The  definition of the set $\pmin$ is motivated by \cite[Definition 4.9]{BL18}.
\begin{defn}\label{defn:APmin}
Suppose  $(G,P)$ is a weak quasi-lattice ordered group with set of atoms $P_a$. Set $P_1 := P_a$ and, recursively, $P_n := P_a^{-1}P_{n-1} \cap P = \{x\inv p \mid x\in P_a, p \in P_n \cap xP\} $
 for each $n \geq 2$; then define
the minimal set \begin{align*}
\pmin := \bigcup_{n \in \N} P_n.
\end{align*}
\end{defn}

\begin{lem}
The minimal set $\pmin$ is the smallest subset of $P$ that contains $P_a$ and is closed under the operations $p\mapsto x^{-1}(x \lor p)$ for  $x \in P_a$.
\end{lem}
\begin{proof}   Clearly the set $\pmin$  contains $P_a$ and is closed under $p\mapsto x^{-1}(x \lor p)$, and if $Q\subset P$ contains $P_a$ and is closed under $p\mapsto x^{-1}(x \lor p)$, then it obviously contains each of the $P_n$, and hence $\pmin$.
\end{proof}

Recall that $\Omega_K$ denotes the intersection $\bigcap_{k\in K} (P\bs kP)$ for each finite subset $K \subset P$; and set $\Omega_\emptyset = P$. Given a nonempty subset $R$ of $P$, we define
$\A_R$ to be the collection  of all \emph{finite (or empty) disjoint unions} of sets $p\Omega_K $:
\begin{equation}\label{eq:sets in A0}
\A_R := \{ \bigsqcup_{j=1}^n p_j\Omega_{K_j} : p_j\in P \text{ and } K_j \subset R \text{ finite or empty}\}.
\end{equation}
By convention $\emptyset \in \A_R$, arising from the empty union.
 It is apparent that if $R\subset T\subset P$, then $\A_{R}\subset \A_T$.
Our strategy for proving reductions of \eqref{eqn:2.1} to a smaller set of inequalities based on $\pmin$ is to show that when $P$ is Noetherian, $\A_{\pmin}$ is itself an algebra of subsets of $P$, namely the algebra $\Bc_P$ from \lemref{lem:algebra of sets general}.

\begin{lem} \label{lem:comp}
The collection
$\A_{P_a} = \{ \sqcup_{j=1}^n p_j\Omega_{K_j} : p_j\in P \text{ and } K_j \subset P_a \text{ finite or empty}\}$ contains  the set $pP$ and its complement $P \setminus pP$
for every $p\in P$.
\end{lem}

\begin{proof} Let $p\in P$.
Setting $K = \emptyset$ in the definition shows that  $pP =p\Omega_\emptyset \in \A_{P_a}$.
 Suppose now that $p = s_{i_1} \dots s_{i_k}$ is an expression for $p$ in terms of atoms.

Then we have
\begin{align*}
P \setminus pP &= (P \bs s_{i_1}P) \: \ml{\sqcup} \: (s_{i_1}P \bs s_{i_1}s_{i_2}P)
\: \ml{\sqcup} \:
 \dots \: \ml{\sqcup} \: (s_{i_1} \: \dots s_{i_{k-1}}P \bs s_{i_1} \dots s_{i_k} P)\\
&= ( P \bs s_{i_1}P) \: \ml{\sqcup} \: s_{i_1}(P \bs s_{i_2}P) \: \ml{\sqcup} \: \dots
\: \ml{\sqcup} \: s_{i_1} \dots s_{i_{k-1}} (P \bs s_{i_k} P),
\end{align*}
which is in $\A_{P_a}$ by  definition.
\end{proof}

\begin{lem}\label{lem:int} Let $(G,P)$ be a weak quasi-lattice ordered group with $P$ Noetherian. Then
%$ \Omega_{K} \in \A_{\pmin}$ for every finite subset $K \subset P$, and
$\A_{\pmin}$ is the algebra  $ \Bc_P$.
\end{lem}

\begin{proof} By \cite[Lemma 2.4]{ALN18} it suffices to prove that $\Omega_K\in \A_{\pmin}$ for every finite subset
$K=\{q_1, \dots, q_n\}$  of $P$. We do this
 by a double induction argument: for an arbitrary $n$, we first apply induction on the number of atoms in $q_i$, and then on
 $k$,  the cardinality of $K \cap P$.\\

Suppose first  $K=\{q_1, \dots, q_n\}$  with $q_1 \in P$ and
$q_2, \dots, q_n \in \pmin$.
We use induction on the number of atoms in a factorisation of $q_1$. If $q_1$ is an atom then $q_1 \in \pmin$, then
$\Omega_K$ is in $\A_{\pmin}$ by definition. Assume now that $\Omega_K\in\A_{\pmin}$ for all $K$ as
above in which $q_1\in P$ is a product of at most $N$ atoms, and let
$q \in P$ be a product of $N+1$ atoms, so that $q = s_{1}q'$ for  $s_1\in P_a$ and
$q'$ a product of $N$ atoms.
Since $\Omega_q = (P \bs s_{1}P) \ml{\dcup} s_{1}(P \bs q'P)$, we have
\begin{equation} \label{eq:her}
\Omega_{\{q,q_2,\dots,q_n\}} = [(P \bs s_{1} P ) \dcup s_{1} (P \bs q' P )] \cap \Omega_{\{q_2,\dots,q_n\}} = \Omega_{\{s_{1},q_2,\dots,q_n\}} \dcup s_{1} (P \bs q' P) \cap  \Omega_{\{q_2,\dots,q_n\}} .
\end{equation}
By definition $\Omega_{\{s_{1},q_2,\dots,q_n\}} \in \A_{\pmin}$, so we only need to show that
$s_{1} (P \bs q' P) \cap  \Omega_{\{q_2,\dots,q_n\}}\in \A_{\pmin}$.
We rewrite this intersection as
\begin{align}\label{eq:difference using s and q}
s_{1} (P \bs q' P) \cap  \Omega_{\{q_2,\dots,q_n\}} &= s_{1} (P \bs q' P) \cap (s_{1} P \bs q_2 P) \cap
\dots \cap (s_{1} P \bs q_n P)\\
&=s_{1} (P \bs q' P) \cap ( s_{1} P \bs (s_{1} \lor q_2) P) \cap
\dots \cap ( s_{1} P \bs (s_{1} \lor q_n) P) \notag\\
&= s_{1} [(P \bs q' P) \cap (P \bs s_{1}^{-1}(s_{1} \lor q_2) P) \cap
\dots \cap (P \bs s_{1}^{-1}(s_{1} \lor q_n) P)].\notag
\end{align}

By the definition of $P_{inf}$ we have that $s_1^{-1}(s_1 \vee q_i) \in P_{inf}$, since $s_1,q_i \in P_{inf}$ for $i = 1, \dots, n$. Moreover, $q'$ is a product of $n$ atoms, thus, by the induction hypothesis, the intersection in the last line above is contained in $\A_{P_{inf}}$.
From this  we can conclude that $\A_{\pmin}$ contains the set
$\Omega_{\{q, q_2,\dots,q_n\}}\in \A_{\pmin}$ whenever $q\in P$ and $q_j\in \pmin$, for $j = 2, \dots, n$.

 If  $K=\{q_1,q_2,\dots,q_n\}$ with $q_1\in P$ and
$q_2,\dots,q_n\in \pmin$, then  $\Omega_K\in \A_{\pmin}$ by the first part of the proof.
Assume now that $\Omega_K\in \A_{\pmin}$ for all $K=\{q_1,q_2,\dots,q_n\}$ such that
$q_1, \dots, q_k \in P$ and
$q_{k+1}, \dots, q_n \in \pmin$. We aim to show that
\begin{equation}\label{eq:induction on n size of K}
\Omega_{K} \in \A_{\pmin}
\qquad \text { when }  q_1, \dots, q_{k+1} \in P \text{  and }
q_{k+2}, \dots, q_n \in \pmin.
\end{equation}
As  in the first part we use induction on the number of atoms in $q_{k+1}$. If
$q_{k+1}$ is an atom itself, then $q_{k+1} \in P_a \subset \pmin$, thus
$\Omega_{\{q_1,\dots,q_{k+1},\dots, q_n\}} \in
\A_{\pmin}$ by the induction hypothesis.
Assume now that \eqref{eq:induction on n size of K} holds true when the $(k+1)$th element  is a product of at most $M$ atoms. Suppose now the $(k+1)$th element is a product of $M+1$ atoms, and write $q_{k+1} = s q' \in P$
with $s \in P_a$ and $q' \in P$  a product of $M$ atoms.
As before,
\[
(P \bs q_{k+1}' P) = (P \bs s P )\ml{\dcup} s(P \bs q' P).
\]
 This splits $\Omega_{\{q_1,\dots,q_{k+1},\dots, q_n\}}$ into a disjoint union of two intersections. The first one has an atom $s$ in the $(k+1)$th term  so it is in $\A_{\pmin}$ by the induction hypothesis, so proving \eqref{eq:induction on n size of K}
 reduces to showing that the second one, namely
$(P \bs q_1 P) \cap \dots \cap s (P \bs q' P) \cap \dots \cap (P \bs q_n P )$ is in $ \A_{\pmin}$. This holds by the second induction hypothesis because
\begin{multline*}
(P \bs q_1 P ) \cap \dots \cap s (P \bs q' P) \cap \dots \cap (P \bs q_n P)
= \\s [ (P \bs s^{-1}(s \lor q_1)P) \cap \dots \cap (P \bs q' P) \cap \dots \cap
(P \bs s^{-1}(s \lor q_n) P))]
\end{multline*}
and $q'$ is the product of $M$ atoms. This concludes the proof.
\end{proof}

We are now ready to prove the main result of this section.

\begin{thm} \label{thm:Pmin}
Let $(G,P)$ be a weak quasi-lattice ordered group with $P$ Noetherian. Assume that we
are given a compactly aligned product system $\{X_p\}_{p \in P}$ of essential
$C^*$-correspondences over $P$ with $X_e = A$, a
homomorphism $N:P \to (0,\infty)$ and $\beta \in \R$. Then a tracial state $\tau$ on
$A$ satisfies the condition of \eqref{eqn:2.1} if and only if
\begin{equation} \label{pmin}
\tau(a) + \sum_{{ \emptyset \neq K \subset J }}
(-1)^{|K|}N(\lor K)^{-\beta}\Tr_{\tau}^{ \lor K}(a) \geq 0 \qquad \text{ for all finite }J \subset \pmin \text{ and }a \in A_+,
\end{equation}
where the terms  corresponding to $\vee K=\infty$ are set to zero.
\end{thm}

\begin{proof}
Obviously \eqref{eqn:2.1} implies \eqref{pmin}.  The proof that
 \eqref{pmin} implies \eqref{eqn:2.1} goes along the lines of the proof of  \cite[Theorem 9.1]{ALN18}, which is the particular case of  right-angled Artin monoids.

 Suppose $\tau$ is a tracial state of $A$ satisfying \eqref{pmin}. By taking $J = \{s\}$ for $s \in P_a$  we see  that
$F_s(\tau) (a) \leq N(s)^{\beta}\tau(a) < \infty$ for all $s \in P_a$ and $a \in A_+$.
By \lemref{lem:Tr-finite-onA-from-S}, we have that $\Tr_{\tau}^p(a)$ is finite for all $p \in P$ and $a \in A_+$.

Next recall that $\mathcal{B}_P$ is the algebra generated by the collection $\{pP: p\in P\}$, and for each  $p\in P$ define
\begin{align*}
\mu(pP) = N(p)^{-\beta}F_p(\tau) \text{ for }p\in P.
\end{align*}
Then $\mu$ extends to a finitely additive measure, which we also denote by $\mu$, defined on $(P,\mathcal{B}_P)$ with values in the finite traces of $A$.
Evaluating $\mu$ at the sets $\Omega_J = \bigcap_{p \in J} ( P \bs p P)$
for finite nonempty subsets $J \subset P\bs\{e\}$ yields
%\begin{align*}
%\mu ( \bigcap_{p \in J} ( P \bs p P))
%= \mu ( (\bigcup_{p \in J}pP )^c )
%= \mu(P) - \mu ( \bigcup_{p \in J} pP )
%\end{align*}
%By using induction on the cardinality of $J$ and the fact that
%$\mu(p_1P \cup p_2 P) = \mu(p_1 P ) + \mu(p_2 P) - \mu (p_1 P \cap p_2 P)$ we get that
\begin{align*}
\mu(\Omega_J)=\mu(P) - \mu ( \bigcup_{p \in J} pP ) =
F_e(\tau) + \sum_{\emptyset \neq K \subset J} (-1)^{|K|}N(\lor K)^{-\beta} F_{\lor K}(\tau).
\end{align*}
Thus our assumption \eqref{pmin} guarantees that $\mu(\Omega_J)\geq 0$ for all finite subsets
$J\subset \pmin$ and we must prove that this positivity holds for all finite subsets $J \subset P\setminus\{e\}$.
By \Cref{lem:int}  it suffices to  show that $\mu$ is positive on all sets in the family $\A_{\pmin}$. Fix a finite subset $K$ of $\pmin$. For $p \in P$ we have that
\begin{align*}
\mu(p \Omega_K) &= N(p)^{-\beta}F_p(\tau) + \sum_{\emptyset \neq H \subset K}
(-1)^{|H|}N(p(\lor H))^{-\beta}F_{p(\lor H)}(\tau) \\
&= N(p)^{-\beta}F_p(\mu(\Omega_K)),
\end{align*}
and the right hand side is positive because $\mu(\Omega_K)$ is positive by assumption. By finite additivity, $\mu$ is positive on every set in $\A_{\pmin}$, which completes the proof.
\end{proof}

\begin{exmp}  By exhibiting classes of examples, we point out that  the set $\pmin$ can be finite, in which case \Cref{thm:Pmin}  provides a significant reduction, and that the inclusion $P_a \subset \pmin$ can be strict for some monoids and an equality for others.

(i) When $P=A_M^+$ is an Artin monoid with finite generating set,  the set $\pmin$ finite because it is contained in the finite Garside family shown to exist in \cite[Theorem 1.1]{DDH15}.

(ii) If $P$ is  a  finite-type Artin monoid that is not abelian, then  there exist canonical generators $s, t$ such that $m_{s,t} > 2$, so that
$s^{-1}(s\lor t) = \la t s\ra^{m_{s,t}-1} \notin S$. Thus $\pmin \neq S$. In this case the arguments of \Cref{sec:tree} yield the  reduction of positivity to generators, which is stronger than \Cref{thm:Pmin}.

(iii) In contrast, when $P$ is a right-angled Artin monoid it is easy to see from \cite[Lemma 9.2]{ALN18} that  $\pmin$ is equal to the set of canonical generators. Formally, this recovers
 \cite[Theorem 9.1]{ALN18}  as a corollary of  \Cref{thm:Pmin}, but unlike  \Cref{pro:finitetree-RAAMcase},
 this proof is not really different since the proof of \thmref{thm:Pmin} is modelled on that of  \cite[Theorem 9.1]{ALN18}.  \end{exmp}

\section{$\KMS$-gaps}\label{sec: gaps}

%\subsection{The positivity criteria for semigroup $C^*$-algebras.}

We assume in this section that $(G,P)$ has no nontrivial invertible elements, that is, $P\cap P\inv =\{e\}$.
 It is known that the associated (full) semigroup $C^*$-algebra $C^*(P)$ may be viewed as the Nica-Toeplitz algebra $\Nica \T (X)$ constructed from the product system $X =\{X_p\}_{p \in P}$
with one-dimensional fibres $X_p = \C$, where the left action is given by complex multiplication for all $p \in P$, see \cite[Proposition 5.6]{SY10}. Suppose that $N:P\to (0,\infty)$, let $\sigma_t(v_p)=N(p)^{it} v_p$ for $p\in P$ and $t\in \R$ and let us analyse \eqref{eqn:2.1} for
$(C^*(P),\sigma)$.

Since $A = X_e \cong \C$,  a tracial state $\tau$ on $A$ is simply $\tau = id$. Moreover, for each finite non-empty $K\subset P$ such that $\vee K<\infty$ we have $\Tr_{\tau}^{\lor K} = \tau = id$. Thus for given $\beta\in \R$, \eqref{eqn:2.1} rewrites as
\begin{equation}\label{eq:2.1 trivial product system}
1 + \sum_{\emptyset \neq K \subset J} (-1)^{|K|} N(\lor K)^{-\beta}\geq 0,  \text{ for all finite nonempty }J\subset P\setminus\{e\}.
\end{equation}
This condition implies the existence of a $\KMS_\beta$-state for
$(C^*(P),\sigma)$ by \thmref{thm:sufficient gauge inv}. This condition also implies that the finitely additive measure on $\Bc_P$ given by
\begin{equation}\label{eq:measure on pP}
\mu(pP)=N(p)^{-\beta} \text{ for }pP\in \Bc_P
\end{equation}
extends to a genuine measure on the $\sigma$-algebra generated by $\{pP:p\in P\}$.

 A refinement of this observation was implicit already in \cite{ALN18}. More precisely, it was shown in \cite[Lemma 7.2]{ALN18} that every finitely additive measure $\mu$ on $(P,\Bc_P)$ such that $\sum_{p\in P}\mu(pP)<\infty$ extends to a genuine measure on the
 $\sigma$-algebra of all subsets of $P$. When $\mu$ is as in \eqref{eq:measure on pP}, we are asking that $\sum_{p\in P}N(p)^{-\beta}<\infty$. Recall from \cite[Definition 7.8]{ALN18} that the infimum of all such real $\beta$ is the \emph{critical inverse temperature} of the system $(C^*(P),\sigma)$, where we assume that $N(p)\geq 1$ for all $p\in P$. When $\beta_c$ is finite, there is a unique
 $\KMS_\beta$-state for every $\beta>\beta_c$, and it is of finite type in the sense of
 \cite[Definition 6.4]{ALN18}, see also \cite[Example 9.6]{ALN18}. Assuming further that $P$ is Noetherian, that $P$ has a finite set of atoms and that $N(p)>1$ for all
 $p\in P\setminus\{e\}$, any possible $\KMS_\beta$-state of infinite type will arise at values $\beta$ where
 \[
 1+\sum_{\emptyset\neq K\subset P_a} (-1)^{\vert K\vert} N(q_K)^{-\beta}=0,
 \]
 cf. \cite{BLRS19} and \cite{ALN18}. Note that this condition means that \eqref{eq:2.1 trivial product system} reduces to having equality at the single subset $J = P_a$ of $P\setminus\{e\}$.

 We note at this point that in case $(G,P)$ is lattice ordered, then for any homomorphism $N:P\to (0,\infty)$, condition \eqref{eq:2.1 trivial product system} is trivially satisfied at $\beta = 0$ for every $J$, because it reduces to $\sum_{K\subset J}(-1)^{\vert K \vert}=0$, which follows by the binomial formula. Thus $(C^*(P),\sigma)$ admits a gauge-invariant
 $\KMS_0$-state. We note that this was first proved in \cite[Proposition 3.7]{BLRS19}, cf. the equivalence of (2) and (4), which does not require the additional assumption that $N(p)=1$ only if $p=e$.
 In this case any $\KMS_0$-state is of infinite type \cite[Corollary 6.10(i)]{ALN18}.

 \subsection{Artin monoids of finite type and $\KMS$-gaps}
Next we wish to apply the results of the previous sections to illustrate a new phenomenon in the context of $C^*$-algebras of Artin monoids, namely the appearance of gaps in the subset of inverse temperatures that support $\KMS$-states. The next definition makes this precise.

\begin{defn}
Let $(G,P)$ be quasi-lattice ordered. Assume that we are given a compactly aligned product
system $\{X_p\}_{p \in P}$ of $C^*$-correspondences over $P$ with $X_e = A$ and a homomorphism $N: P \to (0, \infty)$ .
We say that $(\Nica \T(X),\sigma)$ has a $\operatorname{KMS}$-gap if
\begin{align*}
\{ \beta \in \R \: | \: \text{ there exists a gauge-invariant } \: \operatorname{KMS}_{\beta}\text{-state} \}\text{ is disconnected}.
\end{align*}
\end{defn}

A large class of monoids where there are no  $\KMS$-gaps is provided by non-abelian right-angled Artin monoids, cf. the last paragraph of Example 9.6  of \cite{ALN18}. More precisely, let $P$ be a non-abelian right-angled Artin monoid with  finite generating set $S\subset P\setminus\{e\}$. Let $N:P\to (0,\infty)$ be a homomorphism such that $N(p)=1$ only for $p=e$. Then $P$ is not lattice ordered,  $\beta_c>0$ by \cite[Proposition 3.7]{BLRS19}, and the possible behavior of $\KMS_\beta$-states is as follows: there are none for $\beta<\beta_c$, there is a unique $\KMS_{\beta_c}$-state, which is of infinite type, and  for each $\beta>\beta_c$ there is a unique $\KMS_\beta$-state, which is of finite type. In particular, there are no KMS-gaps.

Next we see that the Artin braid monoids $B_n^+$ for $n\geq 3$ offer a different picture.

\begin{prop}\label{prop:Bn has interval with no KMS}
Suppose $n \geq 3$ and let $G$ be the braid  group $B_n$ with generating
set $\{s_1,s_2,\dots,s_{n-1}\}$ and relations
\[\begin{array}{ll}
 s_i s_j s_i = s_j s_i s_j &\text{ when } \vert i-j\vert =1\\
s_is_j = s_js_i &\text{ when } \vert i-j\vert\geq 2.
\end{array}
\]
Let $B_n^+$ be the associated braid monoid and
let $N:B_n^+\to [1,\infty)$ be the homomorphism given by $N(p)=\exp(\ell(p))$ for $p\in B_n^+$, where $\ell$ is the normalised length function on $B_n^+$. Then $(C^*(B_n^+),\sigma)$ has a gauge-invariant
$\KMS_0$-state and no $\KMS_\beta$-states in the interval $(0,a)$, where $\exp(-a)= \sqrt{5}/2-1/2 \approx 0.61803$ is (the reciprocal of) the golden ratio.
\end{prop}

\begin{proof} Since $P =B_n^+ $ is  lattice ordered, there is a gauge-invariant $\KMS_0$-state, as we already observed. In order to show there is a $\operatorname{KMS}$-gap  that extends from $0$ to at least  $a$ we show that  \eqref{eq:2.1 trivial product system} fails for the subset $J = \{s_1,s_2\}$, that is
\[
g_J(\beta)= 1+ \sum_{\emptyset \neq K \subset J} (-1)^{|K|} N(\vee K)^{-\beta}
\]
is strictly negative  in the interval $(0,a)$.
Since   $\ell(s_1 \lor s_2) = \ell(s_1s_2s_1) = 3$, is is clear that
\[
g_{\{s_1,s_2\}}(\beta)=  1 - 2e^{-\beta} + (e^{-\beta})^3.
\]
The polynomial $1-2t+t^3$ is negative in the interval $ (\frac{\sqrt{5}}{2} - \frac{1}{2},  1)$ determined by its two positive roots, so if we set $t=e^{-\beta}$ we see that $g_{\{s_1,s_2\}}(\beta)<0$ for $\beta\in (0,a)$. Hence \eqref{eq:2.1 trivial product system} fails for $J = \{s_1, s_2\}$ and all $\beta$ in $(0,a)$, so there are no  $\KMS_\beta$-states  in $(0,a)$.
 \end{proof}

\begin{corollary}

\begin{enumerate}
\item Let $t_1=\sqrt{5}/2-1/2 \approx 0.618$ be the smallest positive root of the clique polynomial $1-2t +t^3$ of $B_3^+$, and define $a = -\log t_1$.
Then the inverse temperature space of $(C^*(B_3^+),\sigma)$  is $\{0\} \cup [a,\infty]$

\item Let $r_1$ be the smallest positive root of $1-2t-t^2+t^3+t^4 + t^5$, and $b = -\log r_1$. Then
 the inverse temperature space of $(C^*(B_4^+),\sigma)$  is $\{0\} \cup [b,\infty]$.

\end{enumerate}
\end{corollary}
\begin{proof} The clique polynomial is the reciprocal of the growth series, see \cite{Saito}, so $a$ and $b$ are the critical temperatures for $B_3^+$ and $B_4^+$ respectively and we know from \cite[Theorem 3.5]{BLRS19} and \cite[Proposition 3.7]{BLRS19} that there is a  $\KMS_\beta$-state for each $\beta \in \{0\} \cup [\beta_c,\infty]$, which is unique for $\beta\neq 0$.
Recall also from \cite[Proposition 4.5]{BLRS19} that if there exists a
$\KMS_{\beta}$-state for some $ \beta \in (0, \beta_c)$, then  $e^{-\beta}$ has to be a root of the clique polynomial in the interval $(e^{-\beta_c},1)$;  however, \cite{BLRS19} does not decide whether there are any $\KMS_{\beta}$-states corresponding to any intermediate roots.
Our positivity criterion allows us to retrieve the known part of the temperature space and show that there are no  KMS states in $(0,\beta_c)$.

We deal with $B_3^+$ first. Since here is only one relation, namely $s_1s_2s_1=s_2s_1s_2$,  and all subsets are cliques, the clique polynomial is $1-2t+t^3$, which has roots $-\sqrt{5}/2-1/2<0$, $t_1 = \sqrt{5}/2-1/2$, and $1$ and is strictly positive on $(0,t_1)$. The other cliques have polynomials $1$ and $1-t$, which are  positive on all of $(0,1)$. Hence \eqref{eq:2.1 trivial product system} holds for $\beta\geq \beta_c$, proving the assertion about   $B_3^+$.

 The relations in $B_4^+$ are
  \[
s_1s_2s_1=s_2s_1s_2, \ s_2s_3s_2=s_3s_2s_3\text{ and }s_1s_3=s_3s_1,
  \]
  and since $B_4$ is finite type, all the subsets of $S$ are cliques:
  \[
  \emptyset,\  \{s_1\},\  \{s_2\},\  \{s_3\},\  \{s_1,s_2\},\  \{s_2,s_3\},\  \{s_1,s_3\},\  \{s_1, s_2, s_3\}
   \]
  The common upper bounds of all but the full clique have been computed above, and $\vee \{s_1, s_2, s_3\} = s_3 s_2 s_1s_3 s_2 s_3$, which has length $6$.
Thus, the  clique polynomial for $B_4^+$ %$1 + \sum_{K \subset S}(-1)^{|K|}N(\lor K)^{-\beta}$,
written with  $t=e^{-\beta}$, is $h(t)=1-3t+t^2+2t^3-t^6 = (1-t) (1-2t-t^2+t^3+t^4+t^5)$.

  We may now use the criterion of \eqref{eq:2.1 trivial product system} for cliques $J$ in $S$ to check whether there is a gauge-invariant $\KMS_\beta$-state at $\beta=- \log r_1$.  From the proof of \proref{prop:Bn has interval with no KMS} we know that the condition is trivially satisfied for the choices $J=\{s_1,s_2\}$ and $J=\{s_2,s_3\}$. For the choice $J=\{s_1,s_3\}$, the corresponding polynomial $1-2t+t^2$ is nonnegative everywhere. Finally, the one-element choices of $J$ yield the polynomial $1-t$, which is positive at $t_1$. Hence there is a gauge-invariant $\KMS_\beta$-state at $\beta=-\log r_1$ in view of \corref{cor:directedcase} and \cite[Theorem 5.1]{ALN18}.
In  the interval $(0,1)$ the polynomial $1-2t-t^2+t^3+t^4 + t^5$ has two other roots, $r_2 \approx  0.659$ and $r_3 \approx 0.874$. But since the positivity condition fails at both these roots for the clique polynomial $1-2t +t^3$ of  $J =\{s_1, s_2\}$, there are no  $\KMS_\beta$-state for $\beta \in (0,b)$. This completes the proof for $B_4^+$.
\end{proof}

\section{Characterisation of reduction to generators} \label{sec:alt}
In this section we investigate more closely conditions under which the positivity condition \eqref{pmin} for the minimal set $\pmin$ can be reduced
to a smaller subset of inequalities, namely those involving only finite subsets $J$ of the set $P_a$ of atoms. We assume throughout that $(G,P)$ is a weak quasi-lattice ordered group.
Recall that the key ingredient in the proof of \Cref{thm:Pmin} is \lemref{lem:int} showing that the collection $\A_{\pmin}$ defined via \eqref{eq:sets in A0} contains all sets
$\Omega_J$ for $J \subset P\bs \{e\}$ and hence is the algebra $\Bc_P$ from \cite[Section 2]{ALN18}.
 \begin{thm}\label{thm:redifalg}
If the collection
\[
\A_{P_a} = \{ \sqcup_{j=1}^n p_j\Omega_{K_j} : p_j\in P \text{ and } K_j \subset P_a \text{ finite or empty}\}
\]
 is an algebra, then $\A_{P_a} = \mathcal{B}_P$ and reduction of the  positivity condition  to generators  holds for $P$.
\end{thm}

\begin{proof}
By \Cref{lem:comp}, $pP \in \A_{P_a}$ for all $p \in P$. Now, $\mathcal{B}_P$ is the algebra (closed under finite unions, intersections and complements) generated by the sets $pP$, so if $\A_{P_a}$ is an algebra we get that $\mathcal{B}_P \subseteq \A_{P_a}$, which implies that
$\A_{P_a} = \A_{\pmin}$. Following the proof \Cref{thm:Pmin} we see that condition \eqref{eqn:2.1} can be reduced to subsets of $P_a$.
\end{proof}

In view of this we aim to find conditions ensuring that $\A_{P_a}$ is an algebra. By \lemref{lem:comp} the collection $\A_{P_a}$ contains the sets $pP$ and their  complements $P\setminus pP$, and obviously $\A_{P_a}$ is also closed under finite disjoint unions.

\begin{lem} \label{lem:union}
Let $(G,P)$ be a weak quasi-lattice ordered group with $P$ Noetherian. If  $\A_{P_a}$ is
closed under finite intersections, then it contains all arbitrary finite unions of sets of the form $pP$ with $p \in P$.
\end{lem}

\begin{proof}
The proof is by induction on the cardinality of the union. If the
cardinality is $1$ the union is a set on the form $pP$, which is in $\A_{P_a}$ by \Cref{lem:comp}. Let
$n\geq 1$ and assume that $\A_{P_a}$ is closed under unions of at most $n$ sets of the form $pP$ with $p\in P$. Let $p_1,p_2,\dots, p_{n+1}\in P$. We write
\begin{align*}
\bigcup_{i = 1}^{n+1} p_i P &= \left( \bigcup_{i = 1}^n p_i P \right) \bigcup \; p_{n+1}P \\
&= \left[ \! \left( \bigcup_{i=1}^n p_i P \right)^{\!c} \! \! \bigcap
p_{n+1}P \right] \! \dis \! \left[ \bigcup_{i=1}^n p_i P
\bigcap \: p_{n+1}P \right] \! \dis \! \left[ \bigcup_{i=1}^n p_i P \bigcap
 (p_{n+1} P)^c \right]\\
&= B_1 \; \dis \; B_2 \; \dis \; B_3.
\end{align*}
We show next  that the sets $B_1,B_2$ and $B_3$  are  in $\A_{P_a}$.
We have
$B_1= \bigcap_{i=1}^n (p_i P)^c \bigcap \; p_{n+1}P$, which is in $\A_{P_a}$ by
     \lemref{lem:comp} and our hypothesis that $\A_{P_a}$ is closed under finite intersections. The set
    $B_2$ takes the form
    $ \bigcup_{i = 1}^n \left( p_i P \; \bigcap \; p_{n+1}P \right)
    = \bigcup_{i = 1}^n (p_i \lor p_{n+1}) P$, where $(p_i \lor p_{n+1})P = \emptyset$ if $p_i \lor p_{n+1} = \infty$.
    Since this union has at most $n$ terms, $B_2 \in \A_{P_a}$ by the induction hypothesis. Similarly,  $\bigcup_{i = 1}^n p_i P \in \A_{P_a}$ by the induction hypothesis, and  $(p_{n+1} P)^c \in \A_{P_a}$ by  \lemref{lem:comp}, so $B_3\in \A_{P_a}$ because
     $\A_{P_a}$ is closed under intersections.
     Since $\A_{P_a}$ is closed under disjoint finite unions, the proof is complete.
     \end{proof}

 \begin{prop} \label{prop:alg}
Let $(G,P)$ be a weak quasi-lattice ordered group with $P$ Noetherian. If
$\A_{P_a}$ is closed under finite intersections, then $\A_{P_a} = \Bc_P$.
 \end{prop}

 \begin{proof}
 We start by proving that $\A_{P_a}$ is closed under complements. Fix a set $A\in \A_{P_a}$ of the form $A=\dis_{i\in I}p_i\Omega_{K_i}$, where $I$ is finite, $p_i\in P$ and $K_i\subset P_a$ is finite, for every $i\in I$. We have that
 \[
A^c= \bigcap_{i \in I} \left( p_i \Omega_{K_i} \right)^c = \bigcap_{i \in I} \left[ \left( p_i \bigcup_{k_i \in K_i} k_i P \right)\; \dis \;
 (p_i P)^c \right].
 \]
  Now $(p_i P)^c \in \A_{P_a}$  by \lemref{lem:comp} and
  $\bigcup_{k_i \in K_i} k_i P \in \A_{P_a}$ by \Cref{lem:union}, and it follows that each set in the intersection over $i\in I$ is in $\A_{P_a}$. Since $\A_{P_a}$ is closed under finite intersections, we obtain that $A^c\in \A_{P_a}$, as claimed.
  It remains to show that $\A_{P_a}$ is closed under finite unions. Given $A_1, A_2 \in \A_{P_a}$, we write
 \[
 A_1 \cup A_2 = ( A_1 \cap A_2^c) \dcup (A_1 \cap A_2 ) \dcup (A_2 \cap A_1^c),
 \]
  and use the first part of the proof in conjunction with the hypothesis to obtain that  $A_1 \cup A_2\in \A_{P_a}$. The general case  follows  by induction.
 This proves that $\A_{P_a}$ is an algebra, which is clearly contained in $\Bc_P$ and contains $pP$ for every $p\in P$, so $\A_{P_a} = \Bc_P$.
 \end{proof}

\begin{lem} \label{lem:altcrit} Let $(G,P)$ be a weak quasi-lattice ordered group with $P$ Noetherian. Suppose that
  $(P \bs s P) \cap q\Omega_K \in \A_{P_a}$ for all $s \in P_a$ and  all finite $K \subset \A_{P_a}$. Then
 $\A_{P_a}$ is closed under finite intersections.
 \end{lem}

 \begin{proof} Since every set in $ \A_{P_a}$ is a disjoint union of sets of the form $q\Omega_K$ with $K \subset P_a$, it suffices to prove that $(p\Omega_L) \cap (q\Omega_K) \in \A_{P_a}$ for $p, q\in P$ and $K, L\subset P_a$ finite.

 We proceed by induction on $|L|$. Assume first $L = \{s\}$ for  $s\in P_a$. If $p\vee q=\infty$ we have $pP \cap qP =\emptyset$ and
 $p(P \bs s P) \bigcap q\Omega_K=\emptyset \in \A_{P_a}$. If $p\vee q < \infty$, we have
 \begin{align*}
 p(P \bs s P) \bigcap q\Omega_K &=  p(P \bs s P) \bigcap \bigg( (p \lor q) \bigcap_{k \in K} (P \bs k P) \bigg) \\
 &= p \left[ (P \bs s P)\bigcap \left( p^{-1}(p \lor q)\bigcap_{k \in K} (P \bs k P)
 \right) \right].
 \end{align*}
 Since the intersection is in $\A_{P_a}$ by hypothesis and $\A_{P_a}$ is closed under left multiplication by $p$, this concludes the proof of the case $|L| =1$.

Fix $p,q \in P$ and  a finite $K\subset P_a$. Suppose $(p\Omega_{L'}) \cap (q\Omega_K) \in \A_{P_a}$
 %$\left[ p \bigcap_{h \in L'} (P \bs h P)\right] \bigcap \left[ q \bigcap_{k \in K} (P \bs k P) \right] \in \A_{P_a}$
 for all $L'\subset P_a$ with $|L'| = N$ and let $L \subset P_a$ be such that $L = L' \cup \{t\}$ for $t\in P_a$, so that $|L| = N+1$.  Then
 \begin{align*}
 \left[ p \bigcap_{h \in L} (P \bs h P)\right] \bigcap q\Omega_K &=
 p (P \bs t P) \bigcap \left[ p \bigcap_{h \in L'} (P \bs h P)\right] \bigcap q\Omega_K \\
 &= p (P \bs t P) \bigcap A' \in \A_{P_a}.
 \end{align*}
 By the induction hypothesis  $A' $ is in $\A_{P_a}$, so it is a disjoint union of sets of the form $q_j \Omega_{K_j}$, to which we may  apply the case $|L|=1$  to complete the proof.
 \end{proof}

Next we illustrate the above results with an application to the braid group on two generators.
\begin{lem} \label{lem:cont}
Let $(B_3,B_3^+)$ the finite-type Artin group-monoid pair with generating set $S_2 = \{s_1,s_2\}$ and presentation
\begin{align*}
\la s_1,s_2 \; | s_1s_2s_1 = s_2s_1s_2 \ra.
\end{align*}
If $p \in B_3^+$ can be written as $p = p_1 (s_1s_2s_1) p_2$ for $p_1, p_2 \in P$, then $s_1 \leq p$ and
$s_2 \leq p$.
\end{lem}

\begin{proof}
We only prove that $s_1 \leq p$, as the other claim is analogous. Assume for a contradiction that
$s_1 \nleq p$.  Then we can write $p$ in the form
$p = p_1 s_1s_2s_1 p_2$ where $p_1\neq e$ has the smallest possible length. Now $p_1$ must  end with either $s_1$ or $s_2$. In case $p_1=r_1s_1$,  we have $p=r_1(s_1s_1s_2s_1p_2) = r_1(s_1s_2s_1s_2p_2)$, contradicting the choice of  $p_1$. The case $p_1 = r_1 s_2$  leads to a similar contradiction. Hence $p$ admits an expression in the generators that starts with $s_1$, which precisely means that $s_1 \leq p$.
\end{proof}

\begin{lem} \label{lem:bound}
For all $p \in B_3^+$ and $i=1,2$ we have that
\begin{align*}
p \lor s_i \in \{s_1,\, s_2,\, p,\, ps_1, \,ps_2, \,ps_1s_2, \,ps_2s_1\}.
\end{align*}
\end{lem}

\begin{proof}
Assume $i =1$ first.
The case $p = e$ is trivially satisfied, and the case $p = s_2$ is easy to verify because
$s_1 \vee s_2 =  s_1s_2s_1 = s_2s_1s_2 = p s_1s_2$. The case $s_1 \leq p$ obviously yields $p\vee s_1 = p$. So we may assume that $p$ has at least two letters, and also that $s_1\not\leq p$,   so that $p$ does not have a factor $s_1 s_2 s_1$ by \Cref{lem:cont}. In this case $p\vee s_1 = pb$ for some $b\in B_3^+\setminus \{e\}$. Factoring out the last two letters of $p$, we write $p = p_1x $ where $x$ is one of $s_1 s_2$, $s_2 s_1$, $s_1^2$, or $s_2^2$, and it is easy to see that $b $ is one of $s_1$, $s_2$, $s_2s_1$, or $s_1s_2$, respectively.
\end{proof}

\begin{prop}\label{prop:B3alg}
Let $(B_3,B_3^+)$ be the braid group with generating set
$S_2 = \{s_1,s_2\}$. Then $\A_{S_2}$ is an algebra.
\end{prop}

\begin{proof}
By \Cref{prop:alg} and \Cref{lem:altcrit}, we only need to show that the set
\[
B:=(P \bs s_i P) \bigcap \left[ q \bigcap_{k \in K} (P \bs kP) \right]
\]
is in $\A_{S_2}$ for
$i=1,2$,  finite $K \subseteq S_2$ and $q \in P$. We rewrite
\begin{align*}
B=(P \bs s_i P) \bigcap \left[ q \bigcap_{k \in K} (P \bs kP) \right]
= q\left[(P \bs q^{-1}(q \lor s_i) P) \bigcap \left( \bigcap_{k \in K} (P \bs kP) \right) \right].
\end{align*}
By \Cref{lem:bound}, we have $q \lor s_i \in \{s_i,q,qs_1,qs_2,qs_1s_2,qs_2s_1\}$ for $i = 1,2$. This  yields
\begin{align*}
(P \bs q^{-1}(q \lor s_i)P) \in \{ \emptyset,(P \bs s_1 P),
(P \bs s_2 P), (P \bs s_1s_2 P),(P \bs s_2s_1 P) \}
\end{align*}
for $i=1,2$. It is clear that if $(P \bs q^{-1}(q \lor s_i)P)$ is one of  $\emptyset,(P \bs s_1 P),
(P \bs s_2 P)$, then $B$ is contained in $\A_{S_2}$.

Next we consider the case that
$(P \bs q^{-1}(q \lor s_i) P) = (P \bs s_1s_2 P)$, and examine its possible intersection with $\Omega_K$ for given $K \subset S_2$. If $K = \{s_1\}$, then
    $(P \bs s_1s_2 P) \cap (P \bs s_1 P) =
    (P \bs s_1 P) \in \A_{S_2}$. If $K = \{s_2\}$, then
    $(P \bs s_1s_2 P) \cap (P \bs s_2 P) =
    \left[ s_1(P \bs s_2 P) \dcup (P \bs s_1 P) \right] \cap (P \bs s_2 P)$.
    We need to show that $s_1(P \bs s_2 P) \cap (P \bs s_2 P) \in \A_{S_2}$. This follows from
    \begin{align*}
    s_1(P \bs s_2 P) \cap (P \bs s_2 P) &=
    s_1 \left[( P \bs s_2 P ) \cap (P \bs s_1^{-1}(s_1 \lor s_2)P) \right] \\
    &= s_1\left[(P \bs s_2 P ) \cap (P \bs s_2s_1 P) \right]
    = s_1 (P \bs s_2 P).
    \end{align*}
In case $K = \{s_1,s_2\}$ we obtain $(P \bs s_1s_2 P )\cap (P \bs s_1 P) \cap
    (P \bs s_2 P) = (P \bs s_1 P) \cap (P \bs s_2 P)$, which is in $ \A_{S_2}$.
Checking the remaining case $(P \bs q^{-1}(q \lor s_i) P) = (P \bs s_2s_1 P)$ can be done in the same way. Hence $B\in  \A_{S_2}$ and the proposition follows.
\end{proof}
\begin{rem}
It follows from \Cref{prop:B3alg} that reduction of positivity to generators holds for $B_3^+$, a fact that already follows from \Cref{cor:directedcase}  because $B_3$ is of finite-type.
\end{rem}

New examples of weak quasi-lattice ordered groups satisfying $\A_{P_a} = \Bc_P$ can be obtained by taking free and direct products. Suppose, for instance that $G_1$ and $G_2$ are two Artin groups with canonical generating sets $P_a^1$ and $P_a^2$ and Coxeter matrices $M_1$ and $M_2$, respectively.
 Then their free and direct products are  also Artin groups with generating set $P_a = P^1_a \cup P_a^2$, whose Coxeter matrices have diagonal blocks $M_1$ and $M_2$ and remaining entries equal to $\infty$, in the case of the free product, or $2$, in the case of the direct product. And similarly for the corresponding Artin monoids.

\begin{prop} \label{prop:algebradirect&free}
Let $(G_1,P_1)$ and $(G_2, P_2)$ be two Artin group-monoid pairs, with canonical generating sets  $P_a^1$  and $P_a^2$, respectively, and suppose that
$\A_{P_a^1}$ and $\A_{P_a^2}$ are algebras of subsets of $P_1$ and $P_2$ respectively. Let $P_a = P^1_a \cup P_a^2$.

\begin{enumerate}
\item $(G_1 * G_2, P_1* P_2)$ is an   Artin group-monoid pair with generating set $P_a $ subject to the relations  imposed in the presentations of $G_1$ and $G_2$, and the collection   $\A_{P_a}$ is an algebra of subsets of $P_1* P_2$.

\item $(G_1 \times G_2, P_1\times  P_2)$  is an   Artin group-monoid pair with generating set $P_a $ satisfying the relations $s_1s_2 = s_2s_1$ for all pairs of generators $s_1 \in P_a^1$ and $s_2 \in P_a^2$
in addition to  the relations already imposed in the presentations of $G_1$ and $G_2$, the collection  $\A_{P_a}$ is an algebra of subsets of $P_1\times P_2$.

\end{enumerate}

\end{prop}

\begin{proof}  We prove (1) first. Let $(G,P) $ temporarily denote the free product  to simplify the notation.
By \Cref{lem:altcrit} and \Cref{prop:alg},  it suffices to  show that
\begin{equation}\label{eqn:sufficesfree}
(P \bs t P) \cap q\Omega_K  \in \A_{P_a} \qquad \text{for all} \ t\in P_a, \ q\in P, \ K \subset P_a.
\end{equation}
We start by rewriting $B:=(P \bs t P) \cap q\Omega_K$ as
\begin{equation*}
B=(P \bs t P) \bigcap \left[ q \bigcap_{k \in K} (P \bs kP) \right]
= q\left[(P \bs q^{-1}(q \lor t) P) \bigcap \left( \bigcap_{k \in K} (P \bs kP) \right) \right].
\end{equation*}
Assume first that $t \in P_a^1$. If $q \in P_1$, then $B \in \A_{P_a^1} \subseteq \A_{P_a}$ by the assumption that $\A_{P_a^1}$ forms an algebra
of sets of $P_1$. Otherwise there exists at least one atom $s \in P_a^2$ that is a factor in $q$. Let $s_1$ be the first instance (from the left) of such an atom in $q$.
Then $q = a s_1 b$ where $a \in P_1$ and $b \in P$. We then observe that $t \vee q = q$ if and only if $t$ is a factor in $a$ and $t$ can be shuffled to the
front of $a$, and
that otherwise $t \vee q = \infty$, because $t$ can not be shuffled past $s_1$. This implies that $q^{-1}(q \lor t) \in \{e, \infty \}$, and so
$(P \bs q^{-1}(q \lor t)P) \in \{ P, \emptyset \}$, leading to $B \in \A_{P_a}$. The argument is analogous for the case where $t \in P_a^2$.
This concludes the proof of (1).

In order to prove (2), let now $(G,P) $ denote the direct product. As with (1) it suffices to verify \eqref{eqn:sufficesfree}.
Writing $K =  K_1  \sqcup K_2 $ with  $K_1 \subseteq P_a^1$ and $K_2 \subseteq P_a^2$,
we see that
\begin{align*}
B %&=(P \bs t P) \bigcap \left[ q \bigcap_{k \in K} (P \bs kP) \right]\\
%&= q\left[(P \bs q^{-1}(q \lor t) P) \bigcap \left( \bigcap_{k \in K} (P \bs kP) \right) \right] \\
&= q\left[(P \bs q^{-1}(q \lor t) P) \bigcap \left( \bigcap_{k_1 \in K_1} (P \bs k_1P) \right) \bigcap \left( \bigcap_{k_2 \in K_2} (P \bs k_2P) \right) \right].
\end{align*}

Assume first that $t \in P_a^1$. If $q \in P_1$, then
$q(P \bs q^{-1}(q \lor t) P) \bigcap \left( \bigcap_{k_1 \in K_1} (P \bs k_1P) \right) \in \A_{P_a^1}$ by the assumption that
$\A_{P_a^1}$ forms an algebra
of sets of $P_1$. Otherwise there exists at least one atom $s \in P_a^2$ that is a factor in $q$. Since $P_a^2$ commutes with $P_a^1$ in $P$, we can write $q = q_1 s_1 \cdots s_n$, where
$s_1, \dots, s_n \in P_a^2$ and $q_1 \in P_1$. By assumption $t$ is in $ P_a^1$, so it commutes with all the elements $s_1, \dots, s_n$.  Thus,
 $q \vee t = \infty$ when $q_1 \vee t = \infty$, and $ q \vee t = (q_1 \vee t)s_1 \cdots s_n = s_1 \cdots s_n (q_1 \vee t)$ when  $q_1 \vee t < \infty$, in which case
\begin{align*}
q^{-1} (q \vee t) = q_1^{-1} s_n^{-1} \cdots s_1^{-1} s_1 \cdots s_n (q_1 \vee t) = q_1^{-1}(q_1 \vee t) \in P_1.
\end{align*}
Hence, we either have that $(P \bs q^{-1}(q \vee t)P) = P$ (in the case where $q_1 \vee t = \infty$), or that
\begin{align*}
(P \bs q^{-1}(q \lor t) P) \bigcap \left( \bigcap_{k_1 \in K_1} (P \bs k_1P) \right) = (P \bs q_1^{-1}(q_1 \lor t) P) \bigcap
\left( \bigcap_{k_1 \in K_1} (P \bs k_1P) \right) \in \A_{P_a^1}
\end{align*}
(in the case where $q_1 \vee t < \infty$), by the algebra assumption on $\A_{P_a^1}$.

It remains to prove that $p_1 \Omega_{K_1} \cap \Omega_{K_2} \in \A_{P_a}$ for all $p_1 \in P_1$,
$K_1 \subseteq P_a^1$ and $K_2 \subseteq P_a^2$. We observe that
\begin{align*}
p_1 \Omega_{K_1} \cap \Omega_{K_2}&= p_1 \left[ \bigcap_{k_1 \in K_1} (P \bs k_1 P) \bigcap \bigcap_{k_2 \in K_2} (P \bs p_1^{-1} (p_1 \vee k_2) P ) \right] \\
&= p_1 \left[ \bigcap_{k_1 \in K_1} (P \bs k_1 P) \bigcap \bigcap_{k_2 \in K_2} (P \bs p_1^{-1} p_1 k_2 P ) \right] \\
&= p_1 \left[ \bigcap_{k_1 \in K_1} (P \bs k_1 P) \bigcap \bigcap_{k_2 \in K_2} (P \bs k_2 P ) \right]
\end{align*}
where the second equality follows by the fact that $p_1 \vee k_2 = p_1 k_2 = k_2 p_1$, since $k_2 \in P_a^2 \subseteq P_2$ and $p_1 \in P_1$ and all elements from
$P_1$ commute with all elements from $P_2$. Since the set in the last line of the above calculation is clearly in $\A_{P_a}$, we conclude
$p_1 \Omega_{K_1} \cap \Omega_{K_2} \in \A_{P_a}$.
Interchanging  the roles of $P_a^1$ and $P_a^2$ proves the  case $t \in P_a^2$ and finishes the proof.
\end{proof}

It is easy to see that by taking free and direct products of right-angled or finite-type Artin monoids we obtain other Artin monoids many of which are neither right-angled nor finite-type. As a result we obtain new monoids that satisfy reduction of positivity to generators.

\begin{corollary} \label{thm:NewRed}
Under the assumptions of \Cref{prop:algebradirect&free}, let $(G,P)$ denote either the free or the direct product
of $(G_1,P_1)$ and $(G_2,P_2)$. Assume moreover that we
are given a compactly aligned product system $\{X_p\}_{p \in P}$ of
$C^*$-correspondences over $P$ with $X_e = A$, a
homomorphism $N:P \to (0,\infty)$ and $\beta \in \R$. Then reduction of the positivity condition  to generators holds for $P$.
\end{corollary}
\begin{proof}
Combine \Cref{thm:redifalg} with \Cref{prop:algebradirect&free}.
\end{proof}

\vskip 0.2cm
\emph{Funding}: This work was partially supported by the Natural Sciences and Engineering Research Council of Canada, Discovery Grant RGPIN-2017-04052 to M.L.; the Trond Mohn Foundation through the project ``Pure mathematics in Norway" to M.L.; the Norwegian Mathematical Society through an Abel stipend to L.E.G. to visit the University of Victoria, Canada;   the Cluster of Excellence Mathematics M\"{u}nster at WWU, Germany, through a Research Fellowship to N.S.L; and the RCN grant \#300837 .

\vskip 0.2cm
\emph{Acknowledgements:} M.L. is thankful for the hospitality of the Department of Mathematics at the University of Oslo during a visit in which part of this research was carried out. N.S.L. is thankful to the Cluster of Excellence Mathematics M\"{u}nster  and her host Wilhelm Winter for warm hospitality during a visit in which part of this research was carried out. We especially thank Sergey Neshveyev for inspiring discussions at the early stages of this project and for several helpful remarks towards the end.


\begin{thebibliography}{10}

\bibitem{AaHR18}
Z. Afsar, A. an Huef and I. Raeburn, {\em KMS states on $C^*$-algebras associated to a family of $*$-commuting
local homeomorphisms},  J. Math. Anal. Appl. {\bf 464} (2018), no.~2, 965--1009.

\bibitem{ABCD19}
A. an Huef, B. Nucinkis, C.F. Sehnem, D. Yang {\em Nuclearity of semigroup $C^*$-algebras}, J. Funct. Anal. {\bf 280} (2021), Paper No. 108793.


\bibitem{ALN18}
Z.~{Afsar}, N.~S.~Larsen, and S.~{Neshveyev}, \emph{KMS-states on Nica-Toeplitz $C^*$-algebras}, Comm. Math. Phys. {\bf 378} (2020), no. 3, 1875--1929.


\bibitem{BS72}
E.~Brieskorn and K.~Saito, \emph{Artin-{G}ruppen und {C}oxeter-{G}ruppen}, Invent. Math., {\bf 17}(1972), 245--271.

\bibitem{BLS2} N.~Brownlowe, N.S.~Larsen and N.~Stammeier, \emph{$C^*$-algebras of algebraic dynamical systems and right LCM semigroups}, Indiana Univ. Math. J. {\bf 67} No. 6 (2018), 2453--2486.


\bibitem{BLRS19}
C. Bruce, M. Laca, J. Ramagge and A. Sims, \emph{Equilibrium states and growth of quasi-lattice ordered monoids},  Proc. Amer. Math. Soc. {\bf 147} (2019), no. 6, 2389--2404.


\bibitem{CLSV11}
T.~M. Carlsen, N.~S. Larsen, A.~Sims, and S.~T. Vittadello,
\emph{Co-universal algebras associated to product systems, and
  gauge-invariant uniqueness theorems}, Proc. Lond. Math. Soc. (3), {\bf 103} (2011), no. 4, 563--600.

\bibitem{CCH} P.~Clare, T.~Crisp, N.~Higson, \emph{Adjoint functors between categories of Hilbert modules}, J. Institute of Math. Jussieu  (2016), 1--36.

\bibitem{CL02}
J.~Crisp and M.~Laca,
\emph{On the {T}oeplitz algebras of right-angled and finite-type {A}rtin
  groups}, J. Aust. Math. Soc., {\bf 72} (2002), no. 2, 223--245.


\bibitem{DDH15} P.~Dehornoy, M. Dyer and C. Hohlweg
\emph{Garside families in Artin-Tits monoids and low elements in Coxeter groups},
C.R. Acad. Sci. Paris, Ser. I, {\bf 353}, (2015), 403-408.

\bibitem{Garside-book}
P.~Dehornoy, F.~Digne, E.~Godelle, D.~Krammer and J.~Michel, Foundations of Garside Theory, Number 22 in EMS Tracts in Mathematics, European Math. Soc., 2015.

\bibitem{F02}
N.~J. Fowler, \emph{Discrete product systems of {H}ilbert bimodules},
Pacific J. Math., {\bf 204} (2002), no. 2, 335--375.


\bibitem{HLS}
J.-H.  Hong, N. S. Larsen, and  W. Szymanski, \emph{KMS states on Nica-Toeplitz algebras of product systems},
Internat. J. Math {\bf 23} (12)  125012338 (2012).

\bibitem{KPW} T.~Kajiwara, C.~Pinzari and Y.~Watatani, \emph{Jones index theory for Hilbert $C^*$-bimodules and its equivalence with conjugation theory}, J. Funct. Anal. {\bf 215} (2004) no. 1, 1--49.


\bibitem{KL1}
B. K. Kwa\'{s}niewski and  N. S. Larsen, \emph{Nica-Toeplitz algebras associated with right-tensor $C^*$-precategories over right LCM monoids}, Internat. J. Math.  {\bf 30} No.2 (2019) 1950013.


\bibitem{KL2}
B. K. Kwa\'{s}niewski and  N. S. Larsen, \emph{Nica-Toeplitz algebras associated with product systems over right LCM monoids}, J. Math. Anal. Appl. {\bf 470} (2019) 532--570.


\bibitem{LN04}
M.~Laca and S.~Neshveyev, \emph{KMS states of quasi-free dynamics on {P}imsner algebras},
 J. Funct. Anal., {\bf 211} (2004), no. 2, 457--482.

\bibitem{LR96}
M.~Laca and I.~Raeburn, {\em Semigroup crossed products and the {T}oeplitz algebras of nonabelian groups},
 J. Funct. Anal., {\bf 139} (1996), 415--440.


\bibitem{Law} M.V.~Lawson, \emph{The structure of $0-E$-unitary inverse semigroups I: the monoid case}, Proc.
Edinb. Math. Soc. {\bf 42} (1999), 497--520.


\bibitem{BL18}
B.~{Li},  \emph{Regular dilation and Nica-covariant representation on right LCM semigroups},  Integral Equations Operator Theory {\bf 91} (2019), no. 4, Art. 36, 35 pp.

\bibitem{LL}
M.~Laca and B.~Li, \emph{Amenability and functoriality of right LCM semigroup $C^*$-algebras}, Proc. Amer. Math. Soc., {\bf 128} (2020), no. 12, 5209--5224.

\bibitem{N92}
A. Nica, {\em $C^*$-algebras generated by isometries and Wiener-Hopf operators},
J. Operator Theory, {\bf 27} (1992), 17--52.

\bibitem{OP} D. Olesen  and G.K. Pedersen, \emph{Some
    {$C^{\ast} $}-dynamical systems with a single {KMS} state}, Math. Scand. \textbf{42}
    (1978), no.~1, 111--118.

\bibitem{Paris2002}
L.~Paris.
\emph{Artin monoids inject in their groups},
 Comment. Math. Helv., {\bf 77} (2002), 609--637.

\bibitem{P79}
G.~K. Pedersen,  $C^*$-algebras and their automorphism groups, Pure and Applied Mathematics,  Academic Press, London, 2nd edition, 2018.

\bibitem{Pi97}
M. V. Pimsner, {\em A class of $C^*$-algebras generalizing both Cuntz-Krieger algebras and
crossed product by $\mathbb Z$}, Fields Inst. Commun. {\bf 12} (1997), 189--212.

\bibitem{Saito}
K. Saito, {\em Growth functions for Artin monoids}, Proc. Japan Acad., {\bf 85}, Ser. A (2009), 84--88.

\bibitem{SY10}
A. Sims and T. Yeend, {\em $C^*$-algebras associated to product systems of Hilbert bimodules}, J. Operator Theory {\bf 64} (2010), no.~2, 349--376.



\end{thebibliography}
\end{document}